\documentclass[a4paper,12pt,leqno,twoside]{article}
\usepackage[english]{babel} 
\usepackage{amsmath,amsxtra,amsthm,amssymb}
\usepackage{amsfonts}
\usepackage{latexsym}
\usepackage{xypic}
\usepackage[all]{xy}
\usepackage{stmaryrd} 
\usepackage[mathscr]{eucal}
\usepackage{eufrak}
\usepackage{epsf}
\usepackage[dvips]{graphicx}

\usepackage{epsfig}

\usepackage{graphics}
\usepackage{mathrsfs} 

\usepackage{latexsym}  
\usepackage[usenames]{color}

\usepackage{fancyhdr}
\usepackage{mathrsfs} 

\usepackage{helvet}

\newcommand{\reacy}{$\rea$-sa-constructibility}
\newcommand{\reac}{$\rea$-sa-constructible}
\newcommand{\mc}[1]{\mathcal{#1}}

\newcommand{\mbb}[1]{\mathbb{#1}}
\newcommand{\mrm}[1]{\mathrm{#1}}
\newcommand{\mfr}[1]{\mathfrak{#1}}
\renewcommand{\phi}{\varphi}
\renewcommand{\theta}{\vartheta}
\renewcommand{\rho}{\varrho}
\newcommand{\ep}{\epsilon}
\newcommand{\lra}{\longrightarrow}
\renewcommand{\bigr}[1]{{\big(#1\big)}}
\renewcommand{\Bigr}[1]{{\Big(#1\Big)}}

\newcommand{\ou}[3][]{\overset{{#1}}{\underset{{#2}}{{#3}}}}

\newcommand{\com}{\mathbb{C}}
\newcommand{\rea}{\mathbb{R}}

\newcommand{\integer}{\mathbb{Z}}
\newcommand{\integergz}{\mathbb{Z}_{>0}}
\newcommand{\nat}{\mathbb{N}}
\newcommand{\reagz}{\mathbb{R}_{>0}}

\newcommand{\Mod}{\mathrm{Mod}}

\newcommand{\Op}{\mathrm{Op}}

\newcommand{\opxsac}{\mathrm{Op}^c(\xsa)}

\newcommand{\xsa}{{X_{sa}}}
\newcommand{\ysa}{{Y_{sa}}}

\newcommand{\Ho}[3][]{\mathcal{H}\mathrm{om}_{#1}(#2,#3)}
\newcommand{\mcHom}[3][]{\mathcal{H}\mathrm{om}_{#1}(#2,#3)}
\newcommand{\ho}[3][]{\mathrm{Hom}_{#1}(#2,#3)}

\newcommand{\RH}[3][]{\mathit{R}\mathcal{H}\mathit{om}_{#1}(#2,#3)}

\newcommand{\M}{\mathcal{M}}
\newcommand{\ot}{\mathcal{O}^t}
\newcommand{\otxsa}{\mathcal{O}^t_\xsa}

\newcommand{\dbt}{\mathcal{D}b^t}
\newcommand{\Db}{\mathcal{D}b}

\newcommand{\dbxr}{\mathcal{D}b_{X_\rea}}

\newcommand{\dbtxsa}{\mathcal{D}b^t_{\xsa}}

\newcommand{\D}{\mathcal{D}}
\newcommand{\DX}{{\D_X}}
\newcommand{\dbdx}[1][]{D^b_{#1}(\DX)}
\newcommand{\bdc}[2][]{D^b_{#1}(#2)}
\renewcommand{\mod}{\mathrm{Mod}}
\renewcommand{\M}{\mathcal{M}}
\newcommand{\N}{\mathcal{N}}

\newcommand{\Dfinv}[1]{\mathbb D{#1}^{*}}
\newcommand{\dfinv}[1]{\mathrm D{#1}^{*}}

\newcommand{\Dotimes}{\overset D{\otimes}}
\newcommand{\DDotimes}{\overset {\mbb D}{\otimes}}
\newcommand{\dbd}[2][]{D^b_{#1}(\D_{#2})}
\newcommand{\omt}[1]{\Omega^t_{#1}}
\newcommand{\drt}[2][]{\mrm{D\!R}^t_{{{#1}}}{#2}}
\newcommand{\RGa}[2]{R\Gamma_{[#1]}{#2}}
\newcommand{\SM}{S(\mc M)}

\newcommand{\ch}{\mathrm{char}}
\newcommand{\dmod}[2][]{\mathrm{Mod}_{#1}(\mathcal D_{#2})}

\newcommand{\solt}{\mathscr{S}ol^t}
\newcommand{\THom}{\mc {TH}om}

\newtheorem{thm}{Theorem}[subsection]
\newtheorem{df}[thm]{Def\mbox{}inition}

\newtheorem{prop}[thm]{Proposition}
\newtheorem{lem}[thm]{Lemma}

\newtheorem{rem}[thm]{Remark}

\newtheorem{conj}[thm]{Conjecture}

\numberwithin{equation}{section}
\makeatletter
\@addtoreset{thm}{section}
\makeatother

\newcommand{\ben}{\begin{enumerate}}
\newcommand{\een}{\end{enumerate}}

\author{Giovanni Morando}
\title{\textbf{Constructibility of tempered solutions of holonomic $\D$-modules}}
\date{}

\long\def\symbolfootnote[#1]#2{\begingroup%
\def\thefootnote{\fnsymbol{footnote}}\footnote[#1]{#2}\endgroup}

\newpage

\begin{document}

\maketitle

\thispagestyle{empty}

\begin{abstract}
  In this paper we prove the constructibility on the subanalytic sites
  of the sheaves of tempered holomorphic solutions of holonomic
  $\D$-modules on complex analytic manifolds. Such a result solves a
  conjecture of M. Kashiwara and P. Schapira
  (\cite{ks_microlocal_indsheaves}).

\end{abstract}

\symbolfootnote[0]{\phantom{a}\hspace{-7mm}\textit{2010 MSC.} Primary
  32C38; Secondary 32B20 32S40 14Fxx.}

\symbolfootnote[0]{\phantom{a}\hspace{-7mm}\textit{Keywords and
    phrases:} $\D$-modules, irregular singularities, tempered
  holomorphic functions, subanalytic.}

\vspace{-13mm}

\tableofcontents

\phantom{a}

\phantom{a}

\section*{Introduction}\markboth{Introduction}{Introduction}
\addcontentsline{toc}{section}{\textbf{Introduction}}

The constructibility of the complexes of holomorphic solutions of
holonomic $\D$-modules on complex manifolds, proved by Masaki
Kashiwara in \cite{kashiwara_regular1}, is a fundamental result in the
algebraic study of systems of partial differential equations. In its
general form, it states that if $\mc M,\mc N$ are holonomic
$\D$-modules on a complex manifold $X$, the cohomology groups of the
complex $\RH[\D_X]{\mc M}{\mc N}$ are locally constant sheaves with
finite dimensional stalks when restricted to the strata of some
complex analytic stratification. In particular, this implies (for $\mc
N\simeq\mc{O}_X$) that the complex of holomorphic solutions of a holonomic
$\D$-module on a complex manifold is constructible. A part from its
intrinsic importance, this results is at the base of the statement of
the Riemann--Hilbert correspondence, one of the most important
achievements in $\D$-module theory. Such a correspondence states that
the derived functor of holomorphic solutions establishes an
equivalence between the bounded derived categories of complexes of
$\D$-modules with regular holonomic cohomology and the bounded derived
category of complexes of sheaves with constructible cohomology. In his
proof (\cite{kashiwara_79,kashiwara_riemann-hilbert}), M. Kashiwara
gave an explicit inverse to the functor of holomorphic solutions, the
functor $\mc {TH}om$. Kashiwara defined it on the category of
$\rea$-constructible sheaves. The objects of such a category satisfy
the above conditions of locally constancy and finiteness on some
subanalytic stratification.
For various reasons, the category of $\rea$-constructible sheaves is
more handy than the category of $\com$-constructible sheaves. On the
other hand the $\D$-modules obtained by applying the functor $\mc
{TH}om$ to $\rea$-constructible sheaves which are not
$\com$-constructible are still not well understood.

Later (\cite{ks_moderate_formal_cohomology,ks_indsheaves}),
M. Kashiwara and P. Schapira deepened the study of the functor $\THom$
(and its dual, the Whitney tensor) realizing it as the complex of
sheaves on the subanalytic site $\xsa$ relative to the analytic
manifold $X$ of tempered holomorphic functions, denoted
$\ot_\xsa$. Furthermore, in \cite{ks_microlocal_indsheaves}, they
suggested the use of tempered holomorphic solutions of holonomic
$\D$-modules for approaching the general case of the Riemann--Hilbert
correspondence. Since then, this latter idea produced significant
contributions to the study of irregular $\D$-modules. Let us briefly
recall the classical results in this subject.

The local, 1-dimensional version of the irregular Riemann--Hilbert
correspondence has been solved through the work of many mathematicians
as P.  Deligne, M. Hu\-ku\-ha\-ra, A. Levelt, B. Malgrange,
J.-P. Ramis, Y.  Sibuya, H. Turrittin (see \cite{dmr},
\cite{malgrange_birkhauser} and \cite{sabbah_lnm}). Such a classical
result consists in two steps: the analysis of the formal
decompositions of meromorphic connections and their asymptotic
lifts. The higher dimensional case is more complicated. The formal
decomposition was conjectured by C. Sabbah in \cite{sabbah_ast} and
later proved by T. Mochizuki (\cite{mochizuki1,mochizuki2}) in the
algebraic case and by K. Kedlaya (\cite{kedlaya1}, \cite{kedlaya2}) in
the analytic one. Recently
(\cite{dagnolo_kashiwara_dim1,dagnolo_kashiwara_rh}), A. D'Agnolo and
M. Kashiwara defined the functor of enhanced tempered solutions using
tempered holomorphic functions. They proved that it is fully faithful
on the category of holonomic $\D_X$-modules and they proved a
reconstruction theorem for such category. Such a result is a
generalization of the classical ones as it is of global nature.
Anyway, for the moment being, the statement of a Riemann--Hilbert
correspondence for holonomic $\D$-modules is not clear. Indeed,
despite of the important results characterizing the enhanced tempered
solutions of holonomic $\D$-modules given in
\cite{dagnolo_kashiwara_rh}, a topological description (as for
perverse sheaves or $\com$-constructible sheaves for the regular case)
of the image category of the enhanced tempered solution functor has
not been achieved.

In \cite{ks_microlocal_indsheaves}, the authors give a definition of
$\rea$-constructibility for complexes of sheaves on the subanalytic
sites and they conjectured the $\rea$-constructibility for tempered
solutions of holonomic $\D$-modules. This should be a first attempt to
describe the image category of a functor of tempered solutions. The
conjecture was proved for the case of complex curves
(\cite{morando_existence_theorem}) and in a weaker version for complex
manifold (\cite{morando_preconstructibility}). In the present article
we prove the conjecture in its full generality.

Let us conclude by recalling that Kashiwara--Schapira's conjecture is
strictly related to Kashiwara's functor $\THom$ and the $\D$-modules
obtained by applying it to $\rea$-constructible sheaves. The
conjecture states that the the complex of solutions of these latter
modules (which are not even coherent) with values in any holonomic
$\D$-module is a complex of sheaves on $X$ with $\rea$-constructible
cohomology. 

The present article is organized as follows. In the first section we
review classical results on the subanalytic geometry, the sheaves on
subanalytic sites, the $\D$-modules, the tempered De Rham complexes
and the elementary asymptotic decompositions of flat meromorphic
connections.  In the second section we prove our main result in two
steps, first we prove Kashiwara--Schapira's conjecture for meromorphic
connections, then for $\D$-modules supported on their singular locus.

{\em Acknowledgments}: I wish to express my gratitude to M. Kashiwara,
T. Mochizuki, C. Sabbah and P. Schapira for many essential
discussions. Most of the results exposed in this paper were obtained
during my stay at RIMS, Kyoto University, I acknowledge the kind
hospitality I found there.

\emph{Funding}: the research leading to these results has received
funding from the [European Union] Seventh Framework Programme
[FP7/2007-2013] under grant agreement n. [PIOF-GA-2010-273992].

\section{Notations and review}  \label{recall}
\subsection{Subanalytic sites}

For the theory of subanalytic sets we refer to
\cite{bm_ihes,hironaka_subanalytic_sets,hironaka_pisa}. For the
convenience of the reader we recall here some definitions and the
Rectilinearization Theorem in a case which we will need later.

Let $M$ be a real analytic manifold.

\begin{df}
  \begin{enumerate}
  \item A subset $A$ of $M$ is said \emph{subanalytic at } $p\in M$ if
    there exist a neighborhood $W$ of $p$, a finite set $J$, real
    analytic manifolds $N_{r,j}$ and proper real analytic maps
    $f_{r,j}:N_{r,j}\to W$ ($r=1,2$, $j\in J$) such that
$$ A\cap W= \bigcup_{j\in J}(f_{1,j}(N_{1,j})\setminus f_{2,j}(N_{2,j})) \ . $$
A subset is called \emph{subanalytic} if it is subanalytic at any
point of $M$.

\item A subset $B$ of $\rea^n$ is called a \emph{quadrant} if there
  exists a disjoint partition $\{1,\ldots,n\}=J_0\sqcup J_+\sqcup J_-$
  such that $$ B=\bigg\{(x_1,\ldots,x_n)\in\rea^n;
\scriptsize{  \begin{array}{cc}
x_j=0 & j\in J_0 \\ 
x_k>0 & k\in J_+ \\
x_l<0 & l\in J_-    
  \end{array}}
\bigg\} \ . $$
\end{enumerate}
\end{df}

\begin{thm}[\cite{hironaka_pisa}, Chapters 4 and 7]\label{thm:hironaka_rectilinearization}
  Let $U$ be an open subanalytic subset of a real analytic manifold
  $M$ of dimension $n$ and let $\phi:M\to\rea$ be an analytic map
  which does not vanish on $U$. For any $x_0\in M$ there exist a real analytic map
  $\pi:\rea^n\to X$ and a compact set $K\subset\rea^n$ such that
  \begin{enumerate}
  \item $\pi(K)$ is a neighborhood of $x_0\in M$,
  \item $\pi^{-1}(U)$ is a union of quadrants in $\rea^n$,
  \item $\pi$ induces an open embedding of $\pi^{-1}(U)$ into $X$,
  \item $\phi\circ\pi(y_1,\ldots,y_n)=a\cdot y_1^{r_1}\cdot\ldots\cdot
    y_n^{r_n}$, the $r_j$'s being non-negative integers and
    $a\in\rea\setminus\{0\}$.
  \end{enumerate}
\end{thm}

For the theory of sheaves on topological spaces and for the results on
derived categories we will use, we refer to \cite{ks_som}. For the
theory of sheaves on the subanalytic site that we are going to recall
now, we refer to \cite{ks_indsheaves} and
\cite{prelli_subanalytic_sheaves}.

Let $X$ be a real analytic manifold countable at inf\mbox{}inity. The
subanalytic site $\xsa$ associated to $X$ is defined as follows. An
open subset $U$ of $X$ is an open set for $\xsa$ if it is relatively
compact and subanalytic. The family of open sets of $\xsa$ is denoted
$\opxsac$. For $U\in\opxsac$, a subset $S$ of the family of open
subsets of $U$ is said an open covering of $U$ in $\xsa$ if
$S\subset\opxsac$ and, for any compact $K$ of $X$, there exists a
f\mbox{}inite subset $S_0\subset S$ such that $K\cap(\cup_{V\in
  S_0}V)=K\cap U$.

For $Y=X$ or $\xsa$, $k$ a commutative ring, one denotes by $k_Y$ the
constant sheaf. For a sheaf of rings $\mc R_Y$, one denotes by
$\Mod(\mc R_Y)$ the category of sheaves of $\mc R_Y$-modules on $Y$
and by $\bdc{\mc R_Y}$ the bounded derived category of $\Mod(\mc
R_Y)$.

With the aim of def\mbox{}ining the category $\Mod(k_\xsa)$, the
adjective ``relatively compact'' can be omitted in the
def\mbox{}inition of $\xsa$. Indeed, in \cite[Remark
6.3.6]{ks_indsheaves}, it is proved that $\Mod(k_\xsa)$ is equivalent
to the category of sheaves on the site whose open sets are the open
subanalytic subsets of $X$ and whose coverings are the same as $\xsa$.

One denotes by
$$ \rho:X\lra\xsa \ ,$$ 
the natural morphism of sites given by the inclusion of $\opxsac$ into
the category of open subsets of $X$. We refer to \cite{ks_indsheaves}
for the def\mbox{}initions of the functors
$\rho_*:\Mod(k_X)\lra\Mod(k_\xsa)$ and
$\rho^{-1}:\Mod(k_\xsa)\lra\Mod(k_X)$ and for Proposition
\ref{prop_functors} below.
\begin{prop}\label{prop_functors}
\begin{enumerate}
\item The functor $\rho^{-1}$ is left adjoint to $\rho_*$.
\item The functor $\rho^{-1}$ has a left adjoint denoted by
  $\rho_!:\Mod(k_X)\to\Mod(k_\xsa)$.
\item The functors $\rho^{-1}$ and $\rho_!$ are exact, $\rho_*$ is
  exact on $\rea$-constructible sheaves.
\item The functors $\rho_*$ and $\rho_!$ are fully faithful.
\end{enumerate}
\end{prop}

The functor $\rho_!$ is described as follows. If
$U\in\opxsac$ and $F\in\Mod(k_X)$, then $\rho_!(F)$ is the sheaf on $\xsa$
associated to the presheaf $U\mapsto F\bigr{\overline U}$. 



Let us conclude this subsection by recalling the definition of
$\rea$-sa-constructibility for sheaves on $\xsa$. It is due to
M. Kashiwara and P. Schapira (\cite{ks_microlocal_indsheaves}).

Denote by $\bdc[\rea-c]{\com_X}$ the full triangulated subcategory of
$\bdc{\com_X}$ consisting of complexes whose cohomology modules are
$\rea$-constructible sheaves. In what follows, for
$F\in\bdc{\com_\xsa}$ and $G\in\bdc[\rea-c]{\com_X}$, we set for short
$$ \RH[\com_X]GF:=\rho^{-1}\RH[\com_\xsa]{R\rho_*\,G}F\in\bdc{\com_X} \ .
 $$

\begin{df}\label{df_constructible} 
  Let $F\in\bdc{\com_\xsa}$.  We say that $F$ is \emph{$\rea$
    subanalytic constructible} (\emph{$\rea$-sa-constructible} for
  short) if for any $G\in\bdc[\rea-c]{\com_X}$,
   \begin{equation}
     \label{eq:condition_constructibility}
  \RH[\com_X]GF\in\bdc[\rea-c]{\com_X}\ . \end{equation}
\end{df}

Let us remark that the condition
\eqref{eq:condition_constructibility}, defining the property of
\reacy, implies that such property is local with respect to the
topology of $\xsa$ and that it can be tested locally on $X$ (i.e. with
respect to $G$).

\begin{rem}
  The definition of $\rea$-sa-constructibility on $\xsa$ given above is
  quite abstract. A more geometrical description of such property,
  recalling the classical one for sheaves on analytic manifolds, would
  be of great interest.

  With this aim in mind, it is worth recalling that, by means of
  classical topos theory, one can prove the existence of a topological
  space $\widetilde {X_{sa}}$ such that $\Mod(\com_{\widetilde
    {X_{sa}}})$ is equivalent to $\Mod(\com_{X_{sa}})$. The
  semi-algebraic case is deeply studied in \cite{bcr}. Anyway, the
  different descriptions of $\widetilde {X_{sa}}$ do not allow a
  straightforward generalization of the classical constructibility
  property.
\end{rem}

\subsection{$\D$-modules}

The results on $\D$-modules we are going to use in this paper are well
exposed in the literature, see for example \cite{kashiwara_dmod} and
\cite{bjork}. Nonetheless, for the convenience of the reader, we
prefer to recall them here.

Let $X$ be a complex analytic manifold. One denotes by $\mathcal{O}_X$
the sheaf of rings of holomorphic functions on $X$ and by $\D_X$ the
sheaf of rings of linear partial differential operators with
coefficients in $\mathcal{O}_X$.

Given two left $\D_X$-modules $\mc M_1,\mc M_2$, one denotes by $\mc
M_1\Dotimes\mc M_2$ the internal tensor product and by
$\cdot\overset{\mbb D}\otimes\cdot$ its extension to the derived
category of $\D_X$-modules. Let us start by recalling the following

\begin{prop}
\label{prop:tensor_commutation}
Let $\mc N\in\bdc{\D_X^{op}}$, $\mc M_1, \mc M_2\in\bdc{\D_X}$. 
 Then
$$  \mc N\ou[L]{\D_X}\otimes(\mc M_1\overset{\mbb D}\otimes\mc M_2)\simeq (\mc N\ou[L]{\mathcal{O}_X}\otimes \mc M_1)\ou[L]{\D_X}\otimes \mc M_2  \ .$$  
\end{prop}

Now, let $T^*X$ denote the cotangent bundle on $X$. We denote by
$\dmod[coh] X$ the full subcategory of $\dmod X$ whose objects are
coherent over $\D_X$. For $\M\in\dmod[coh] X$ we denote by $\ch\M$ the
characteristic variety of $\M$. Recall that $\ch\M\subset T^*X$ and
that $\mc M$ is said \emph{holonomic} if $\ch\M$ is Lagrangian. One
denotes by $\dmod[h] X$ the full subcategory of $\dmod X$ consisting
of holonomic modules.

\sloppy One denotes by $\bdc[coh]{\D_X}$ (resp. $\bdc[h]{\D_X}$) the
full subcategory of $\dbdx$ consisting of bounded complexes whose
cohomology modules are coherent (resp. holonomic) $\D_X$-modules. For
$\M\in\dbdx[coh]$, set $ \ch\M:= \cup_{j\in\integer}\ch H^j(\M)$.

Let $\pi_X:T^*X\to X$ be the canonical projection, $T^*_XX$ the zero
section of $T^*X$ and $\dot T^*X:=T^*X\setminus T^*_XX$.

For $\M\in\bdc[coh]{\D_X}$, the \emph{singular locus} of $\mc M$ is defined as
$$ S(\M):=\pi_X\Bigr{\ch\M\cap\dot T^*X} \ . $$

It is well known that, if $\M\in\dbdx[h]$, then $S(\M)\neq X$ is a
closed analytic subset of
$X$.

Now, we are going to recall the definition of a regular holonomic
$\D$-module. There are several equivalent definitions (see
\cite[Definition 5.2 and Proposition 5.5]{kashiwara_dmod}), we chose
the following one because, in our opinion, it is the most direct and
easy to state.

\begin{df}
An object $\M\in\dbdx[h]$ is said \emph{regular holonomic} if, for any $x\in X$,
$$ \mathrm{RHom}_{\D_X}(\M,\mathcal{O}_{X,x})\overset{\sim}{\lra} \mathrm{RHom}_{\D_X}(\M,\widehat{\mathcal{O}}_{X,x})  \
, $$ where $\widehat{\mathcal{O}}_{X,x}$ is the $\D_{X,x}$-module of
formal power series at $x$. We denote by $\dbdx[rh]$ the full
subcategory of $\dbdx[h]$ consisting of $\mc M\in\dbdx[h]$ all of
whose cohomology modules are regular.
\end{df}

Now, let $Z$ be a closed analytic subset of $X$. Let $\mc I_Z$ be the
coherent ideal consisting of the holomorphic functions vanishing on
$Z$, for $\mc M\in\dmod X$, one sets
\begin{eqnarray*}
  \Gamma_{[Z]}\mc M&:=&\ou k\varinjlim\ \mcHom[\mathcal{O}_X]{\mathcal{O}_X/\mc I^k_Z}{\mc M} \ ,\\
  \Gamma_{[X\setminus Z]}\mc M&:=&\ou k\varinjlim\ \mcHom[\mathcal{O}_X]{\mc I^k_Z}{\mc M} \ .
\end{eqnarray*}

If $S\subset X$ can be written as $S=Z_1\setminus Z_2$, for $Z_1$ and
$Z_2$ closed analytic sets, then it can be proved that the following
object is well defined

$$ \Gamma_{[S]}\mc M:=\Gamma_{[Z_1]}\Gamma_{[X\setminus Z_2]}\mc M  $$
and that $\Gamma_{[S]}$ is a left exact functor. One denotes by
$R\Gamma_{[S]}$ the right derived functor of $\Gamma_{[S]}$.

\begin{thm}\label{thm:RGamma}
  Let $Z$ be a closed analytic subset of $X$, $S_1,S_2$ differences of
  closed analytic subsets of $X$, $\mc M\in\dbdx$.
  \begin{enumerate}
  \item The following is a distinguished triangle in $\bdc{\D_X}$
 $$ R\Gamma_{[Z]}\mc M\lra\mc M\lra R\Gamma_{[X\setminus Z]}\mc
 M\overset{+1}\lra \ . $$
\item We have $\RGa Z{\mc M}\simeq \RGa Z{\mc {O}}\DDotimes\mc M$ and $\RGa
  {X\setminus Z}{\mc M}\simeq \RGa {X\setminus Z}{\mathcal{O}}\DDotimes\mc M$.
\item We have $\RGa{S_1}{\RGa{S_2}{\mc M}}\simeq \RGa{S_1\cap S_2}{\mc
    M}$.
  \end{enumerate}
\end{thm}

Given $f\in\mathcal{O}_X$, let $Z:=f^{-1}(0)$. One denotes by
$\mathcal{O}_X[*Z]$ the sheaf of meromorphic functions with poles on
$Z$. Let us remark that $\mathcal{O}_X[*Z]$ is flat over
$\mathcal{O}_X$. Given $\mc M\in\Mod(\D_X)$, 
one can prove that $R\Gamma_{[X\setminus Z]}\mc M\simeq \mc
M\overset{D}\otimes\mathcal{O}_X[*Z] $.

Given two complex analytic manifolds $X$, $Y$ and a holomorphic map
$f:Y\to X$, one denotes by $\dfinv f:\dmod X\to\dmod Y$ the inverse
image functor and by $\Dfinv f:\dbd X\to \dbd Y$ the derived
functor. Moreover, one has $\Dfinv f:\dbd[h]X\to \dbd[h] Y$. Recall
that $f$ is said \emph{smooth} if the corresponding maps of tangent
spaces $T_yY\to T_{f(y)}X$ are surjective for any $y\in Y$. If $f$ is
a smooth map, then $\dfinv f$ is an exact functor.

We conclude describing the behavior of $R\Gamma_{[X\setminus Z]}$
with respect to $\Dfinv{f}$. Proposition
\ref{prop:temp_supp_inv_image} below can be directly obtained using,
for example, Proposition 2.5.27 and Theorem 2.3.17 of \cite{bjork}.

\begin{prop}\label{prop:temp_supp_inv_image}
  Let $f:Y\to X$ be a holomorphic map, $Z$ an analytic subset of $X$,
  $\mc M\in\dbd[h]X$. Then,
$$  \Dfinv f R\Gamma_{[X\setminus Z]} \mc M\simeq R\Gamma_{[Y\setminus f^{-1}(Z)]} \Dfinv f\mc M  \ .$$
\end{prop}

\subsection{Tempered De Rham complexes}\label{subsection:tempered_de_rham}

We start this subsection by recalling some results of
\cite{ks_indsheaves} on tempered holomorphic functions and tempered De
Rham complexes of holonomic $\D$-modules. Then we prove some general
results on the \reacy\ of the tempered De Rham complex with respect
to the inverse image functors.

Given a complex analytic manifold $X$ of dimension $d_X$, one denotes
by $X_\rea$ the real analytic manifold underlying $X$. Furthermore,
one denotes by $\overline X$ the complex conjugate manifold, in
particular $\mathcal{O}_{\overline X}$ is the sheaf of
anti-holomorphic functions on $X$. Denote by $\Db_{X_\rea}$ the sheaf
of distributions on $X_\rea$ and, for a closed subset $Z$ of $X$, by
$\Gamma_Z(\Db_{X_\rea})$ the subsheaf of sections supported by
$Z$. One denotes by $\dbtxsa$ the presheaf of \emph{tempered
  distributions} on $\xsa$ def\mbox{}ined by
$$\Op^c(\xsa)^{op}\ni U \longmapsto \dbtxsa(U):=\Gamma(X;\dbxr)\big/\Gamma_{X\setminus U}(X;\dbxr) \ .$$
In \cite{ks_indsheaves} it is proved that $\dbtxsa$ is a sheaf on $\xsa$. This
sheaf is well def\mbox{}ined in the category $\mod(\rho_!
\D_X)$. Moreover, for any $U\in\opxsac$, $\dbtxsa$ is
$\Gamma(U,\cdot)$-acyclic.

The sheaf $\dbtxsa$ is strictly related to the Kashiwara's
$\THom(\cdot,\Db_X)$ functor introduced in \cite{kashiwara_79} and
deeply studied \cite{kashiwara_riemann-hilbert}. For the definition
see \cite[Definition 3.13]{kashiwara_riemann-hilbert}. Let us simply
recall \cite[Proposition 7.2.6 \emph{(i)}]{ks_indsheaves}

 \begin{equation}
   \label{eq:dbt_thom}
    R\Ho{F}{\dbtxsa}\simeq \THom(F,\Db_X) \ ,
 \end{equation}
for $F\in\bdc[\rea-c]{\com_X}$.

One def\mbox{}ines the complex of sheaves $\ot_\xsa \in D^b\bigr{\rho_!\D_X}$ 
of tempered holomorphic functions as
\begin{equation}
  \label{eq_def_ot}
  \ot_\xsa:=R\mathcal Hom_{\rho_! \mathcal{D}_{\overline X}}\bigr{\rho_!\mathcal{O}_{\overline X},\dbtxsa}\ .
\end{equation}

It is worth to mention that, if $\dim X=1$, then $  \ot_\xsa $ is
concentrated in degree $0$ and, for $U\in\Op^c(\xsa)$, we have
\begin{equation*}
   \otxsa(U)\simeq\{u\in\mathcal{O}_X(U);\ \exists C,N>0,\ \forall x\in U,\ |u(x)|\leq C\mrm{dist}(x,\partial U)^{-N}\}\ .
\end{equation*}

We also have \cite[Proposition 7.3.2]{ks_indsheaves}
 \begin{equation}
   \label{eq:ot_thom}
    R\Ho{F}{\ot_\xsa}\simeq \THom(F,\mathcal{O}_X) \ ,
 \end{equation}
for $F\in\bdc[\rea-c]{\com_X}$.

Let $\Omega^j_X$ be the sheaf of differential forms of degree $j$. For
sake of simplicity, let us write $\Omega_X$ instead of
$\Omega^{d_X}_X$.

Set
\begin{eqnarray*}
  \omt X & := & \rho_!\Omega_X\ou{\rho_!\mathcal{O}_X}\otimes \otxsa \ .
\end{eqnarray*}

\begin{thm}[\cite{ks_indsheaves} Theorem 7.4.12, Theorem 7.4.1]\label{thm:indsheaves}
  \begin{enumerate}
  \item  Let $\mc L\in\dbdx[rh]$ and set $L:=\RH[\D_X]{\mc L}{\mathcal{O}_X}$. There
  exists a natural isomorphism in $\bdc{\com_\xsa}$
$$  \omt X\ou [L]{\rho_!\mathcal{O}_X}{\otimes}\rho_! \mc L\simeq\RH[\com_\xsa]L{\omt X}  \ . $$
\item Let $X,Y$ be complex manifolds of dimension, respectively, $d_X$
  and $d_Y$; $f:Y\to X$ a holomorphic map and let $\N\in\dbd
  X$. There is a natural isomorphism in $D^b(\com_\ysa)$ 

$$ \omt Y\ou
[L]{\rho_!\D_Y}{\otimes}\rho_! (\Dfinv f \mc N)[d_Y]\overset\sim\lra
f^!(\omt X\ou [L]{\rho_!\D_X}{\otimes}\rho_! \mc N)[d_X] \ .
$$
  \end{enumerate}
\end{thm}

For $\mc M\in\bdc{\D_X}$, we set for short

$$
\begin{array}{rcll}
\drt [X]{\mc M } & := & \omt X\ou [L]{\rho_!\D_X}{\otimes}\rho_! \mc M[-d_X]&\in\bdc{\com_\xsa}\ ,\\
 \solt(\M) & := & \RH[\rho_!\D_X]{\rho_!\mc M}{\ot_\xsa}&\in\bdc{\com_\xsa} \ .
\end{array}
$$

For $\mc M\in\dmod[h]X$, set $ \mathbb D_X\mc M:=\mc Ext^{d_X}_{\D_X}(\mc
M,\D_X\otimes_{\mathcal{O}_X}\Omega_X^{\otimes-1}) \ . $

\begin{prop}
  \begin{enumerate}
  \item The functor $\mathbb D_X:\dmod[h]X^{op}\to\dmod[h]X$
    is an equivalence of categories.
  \item Let $\mc M\in\dbdx [h]$. Then,
    \begin{equation}
      \label{eq:solt_drt}
      \solt(\mc M)\simeq\drt[X]{(\mathbb D_X\mc M)}  \ .    
    \end{equation}
\end{enumerate}  
\end{prop}

We are now going to recall a conjecture of M. Kashiwara and
P. Schapira on the $\rea$ subanalytic constructibility of tempered
holomorphic solutions of holonomic $\D$-modules and the result we
obtained on curves.

\begin{conj}[\cite{ks_microlocal_indsheaves}]\label{conj:constr} Let
  $\mc M\in\bdc[h]{\D_X}$. Then $\solt(\M)\in\bdc{\com_\xsa}$ is $\rea$-sa-constructible.
\end{conj}

In \cite{morando_existence_theorem}, we proved that Conjecture
\ref{conj:constr} is true on analytic curves.

\begin{thm}\label{thm:r-c_dim1}
Let $X$ be a complex curve and $\M\in\bdc[h]{\D_X}$. Then, 
$\solt(\mc M)$ is $\rea$-sa-constructible.
\end{thm}

Later we will use a version of the conjecture using the functor
$\THom$ instead of the complex of sheaves $\ot$. Let us show how to
pass from one to the other with the following easy

\begin{lem}
 For $G\in\bdc[\rea-c]{\com_X}$ and $\mc
M\in\bdc[h]{\D_X}$ we have
\begin{equation}    \label{eq:solutions_ot_thom}
  R\Ho[\com_\xsa]{G}{\solt(\mc M)} \simeq  R\Ho[\D_X]{\mc M}{\THom_{\com_\xsa}(G,\mathcal{O}_X)} \ .
\end{equation}
\end{lem}
\begin{proof}
We have the following sequence of isomorphisms
\begin{equation*}
  \rho^{-1}R\Ho[\com_\xsa]{\rho_*G}{R\Ho[\rho_!\D_X]{\rho_!\mc M}{\ot_\xsa}} \simeq 
\qquad\qquad\qquad\qquad\qquad\qquad\phantom{a}
\end{equation*}
\vspace{-10mm}
\begin{eqnarray*}
\phantom{a}\qquad\qquad\qquad\qquad\qquad & \simeq & 
\rho^{-1}R\Ho[\rho_!\D_X]{\rho_*G\otimes_{\com_\xsa}\rho_!\mc M}{\ot_\xsa} \\
&\simeq &\rho^{-1}R\Ho[\rho_!\D_X]{\rho_!\mc M}{R\Ho[\com_\xsa]{\rho_*G}{\ot_\xsa}}\\
&\simeq &R\Ho[\D_X]{\mc M}{\rho^{-1}R\Ho[\com_\xsa]{\rho_*G}{\ot_\xsa}} \\
&\simeq & R\Ho[\D_X]{\mc M}{\THom_{\com_\xsa}(G,\mathcal{O}_X)} \ ,
\end{eqnarray*}
where we used \cite[Proposition 1.1.16]{prelli_subanalytic_sheaves} in
the third isomorphism.
\end{proof}

As it will be useful later, let us give an explicit form to
$R\Ho[\com_X]{\com_U}{\drt [X]{\mc M}}$, for $U\in\Op^c(\xsa)$, $\mc
M\in\dmod[h]X$.

Recall that one denotes by $X_\rea$ the real analytic manifold
underlying $X$, by $\mc C^\omega_{X_\rea}$ the sheaf of real analytic
functions on $X_\rea$, by $\overline X$ the complex conjugate manifold
of $X$, by $\mathcal{O}_{\overline X}$ (resp. $\mc D_{\overline X}$,
$\Omega_{\overline X}$) the sheaf of holomorphic functions
(resp. linear differential operators with coefficients in
$\mathcal{O}_{\overline X}$, maximal degree differential forms) on $\overline
X$. Furthermore, we denote by $\Omega_X^\bullet$
(resp. $\Omega^{\bullet,\bullet}_{X_\rea}$) the complex (resp. double
complex) of sheaves of differential forms on $X$ (resp. $X_\rea$) with
coefficients in $\{\mathcal{O}_X$ (resp. $\mc C^\omega_{X_\rea}$).

Now, since

\begin{eqnarray*}
  R\Ho[\com_X]{\com_U}{\drt [X]{\mc M}} & \simeq & R\Ho[\D_{X_\rea}]{\mbb D\mc M\underset{\mathcal{O}_X}\otimes\mc C^\omega}{\THom(\com_U,\Db_{X})}  \\
  & \simeq &  \THom(\com_U,\Db_{X})\underset{\mc C^\omega}\otimes\Omega_{X_\rea}^{\bullet,\bullet}\underset{\mathcal{O}_X}\otimes\mbb D\mc M \ ,
\end{eqnarray*}
we have that $R\Ho[\com_X]{\com_U}{\drt [X]{\mc M}}$ is isomorphic to
the total complex relative to the double complex

\begin{equation}\label{eq:de_rham_explicit}
  \xymatrix{
& 0\ar[d] & 0\ar[d] &
\\
0\ar[r] 
&  \THom(\com_U,\Db_{X})\ou{\mc C^\omega_{X_\rea}}{\otimes} \Omega_{X_\rea}^{0,0}\ou{\mathcal{O}_X}{\otimes}\mbb D\mc M\ar[r]^{\nabla} \ar[d]^{\overline\partial}
& \THom(\com_U,\Db_{X})\ou{\mc C^\omega_{X_\rea}}{\otimes} \Omega_{X_\rea}^{1,0}\ou{\mathcal{O}_X}{\otimes}\mbb D\mc M\ar[r]^{\phantom{asfdasdfasdfaaa}\nabla} \ar[d]^{\overline\partial}
& \ldots
\\
0\ar[r] 
& \THom(\com_U,\Db_{X})\ou{\mc C^\omega_{X_\rea}}{\otimes} \Omega_{X_\rea}^{0,1}\ou{\mathcal{O}_X}{\otimes}\mbb D\mc M\ar[r]^{\nabla} \ar[d]^{\overline\partial}
& \THom(\com_U,\Db_{X})\ou{\mc C^\omega_{X_\rea}}{\otimes} \Omega_{X_\rea}^{1,1}\ou{\mathcal{O}_X}{\otimes}\mbb D\mc M\ar[r]^{\phantom{asfdasdfasdfaaa}\nabla} \ar[d]^{\overline\partial}
& \ldots
\\
& \ldots &\ldots &   
}
\end{equation}

where the $\nabla$ above, in a local coordinate system $z:V\to \com^{n}$, is defined by 
$$\nabla(f\otimes\omega\otimes m)=\ou[n]{j=1}{\sum}\partial_{z^j}f\otimes dz^j\wedge\omega\otimes m+f\otimes d\omega\otimes m+f\otimes\ou[n]{j=1}\sum dz^j\wedge\omega\otimes \partial_{z^j}m  $$
for $f\in\dbtxsa$, $\omega\in\Omega^{k,h}_{X_\rea}$ and $m\in\mbb D\mc M$.

Now, we prove some general results on the behaviour of the \reacy\ of
a tempered De Rham complex with respect to the tensor product and the
inverse image functor. We obtained similar results in
\cite{morando_preconstructibility}, we adapted the proofs to the
present case.

\begin{lem}
  \label{lem:regular_twist}
Let $\mc M\in\dbdx[h]$ and $\mc R\in\dbdx[rh]$.  Set
  $L:=R\Ho[\D_X]{\mc R}{\mathcal{O}_X}\in\bdc[\rea-c]{\com_X}$. For
  any $G\in\bdc[\rea-c]{\com_X}$, one has that
  \begin{equation}
    \label{eq:standard_passage}
    R\Ho[\com_X]{G}{\drt[X]{(\mc M\overset{\mbb D}\otimes\mc R)}}  \simeq  R\Ho[\com_X]{G\ou{\com_X}\otimes L}{\drt[X]{\mc M}}
  \end{equation}

  In particular, if $\drt [X]{\mc M}$ is $\rea$-sa-constructible, then
  $\drt[X]{(\mc M\overset{\mbb D}\otimes\mc R)}$ is
  $\rea$-sa-constructible.

\end{lem}
\begin{proof} 
 
 Let
  $G\in\bdc[\rea-c]{\com_X}$, the following sequence of
  isomorphisms proves \eqref{eq:standard_passage}
 \begin{eqnarray*}
   R\Ho[\com_X]{G}{\omt X\ou[L]{\rho_!\D_X}\otimes \rho_!(\mc M\overset{\mbb D}\otimes\mc R)}  &\simeq &   R\Ho[\com_X]{G}{(\omt X\ou[L]{\rho_!\mathcal{O}_X}\otimes\rho_!\mc R)\ou[L]{\rho_!\D_X}\otimes\rho_!\mc M} \\
 &\simeq &   R\Ho[\com_X]{G}{R\Ho[\com_\xsa]L{\omt X}\ou[L]{\rho_!\D_X}\otimes\rho_!\mc M} \\
 &\simeq &   R\Ho[\com_X]{G}{R\Ho[\com_\xsa]L{\omt X\ou[L]{\rho_!\D_X}\otimes\rho_!\mc M}} \\
 &\simeq &   R\Ho[\com_X]{G\ou{\com_X}\otimes L}{\omt X\ou[L]{\rho_!\D_X}\otimes\rho_!\mc M}  \ .\\
  \end{eqnarray*}

In the previous series of isomorphisms we used Proposition
  \ref{prop:tensor_commutation} in the first isomorphism and Theorem
  \ref{thm:indsheaves}(i) in the second isomorphism.
\end{proof}

Let us recall two definitions.

\begin{df}\label{df:modification_ramification}
  \begin{enumerate}
  \item 
Given a closed analytic set $Z\subset X$, a morphism $f:Y\to X$
  of analytic manifolds is said to be a \emph{modification with
    respect to }$Z$, or simply a modification, if it is proper and if
  $f|_{Y\setminus f^{-1}(Z)}$ is an isomorphism on $X\setminus Z$.
\item Suppose that $X\simeq\com^n$ and $Z$ is defined by the
equation $x_1\cdot\ldots\cdot x_k=0$. 
 By a \emph{ramification map fixing} $Z$ we mean a map
\begin{equation*}
      \begin{array}{rccl}
        \rho_l: & \com^n & \lra & \com^n \\
        & (t_1\ldots,t_n) & \longmapsto & (t_1^l\ldots,t_k^l,t_{k+1}\ldots,t_n) \ ,
      \end{array}
    \end{equation*}
    for some $l\in\integergz$.
  \end{enumerate}  
\end{df}

\begin{prop}\label{prop:constructibility_inv_im}
  \begin{enumerate}
  \item Given $Z\neq X$ a closed analytic hypersurface of $X$, $f:Y\to
    X$ a modification with respect to $Z$, $\mc M\in\bdc{\D_X}$ such
    that $R\Gamma_{[X\setminus Z]}\mc M\simeq\mc M$ and
    $G\in\bdc[\rea-c]{\com_X}$, we have 
    \begin{multline}
      \label{eq:inverse_image_modification}
    R\Ho[\com_X]{G}{\drt[X]{\mc M}}
 \simeq\\ 
\simeq Rf_*  R\Ho[\com_{Y}]{f^{-1}(G\otimes\com_{X\setminus Z})}{\drt[Y]{\Dfinv f\mc M}}
\ . 
    \end{multline}
    In particular, if $\drt[Y]{\Dfinv f\mc M}$ is \reac, then
    $\drt[X]{\mc M}$ is \reac.
\item Suppose we are in the situation of Definition
  \ref{df:modification_ramification} (ii).  Let $f:Y\to X$ be a
  ramification fixing $Z$, $\mc M\in\bdc{\D_X}$ such that
  $R\Gamma_{[X\setminus Z]}\mc M\simeq\mc M$, $U\in\Op^c(\xsa)$. There exists a local system $L$ on $X\setminus Z$ such that the following isomorphism holds 

  \begin{multline}    \label{eq:inverse_image_ramification}
        Rf_*R\Ho[\com_{Y}]{f^{-1}(\com_{U\setminus Z})}{\drt[Y]{\Dfinv f\mc M}}    
\simeq\\
\simeq R\Ho[\com_X]{\com_U\oplus L}{ \drt[X]{\mc M}}  \ .
  \end{multline}
In particular, if $\drt[Y]{\Dfinv f\mc M}$ is \reac, then
    $\drt[X]{\mc M}$ is \reac.
  \end{enumerate}
\end{prop}

\begin{proof}
\  \emph{(i)} Let $f:Y\to X$ be a modification. The following sequence
  of isomorphisms proves the statement
\begin{eqnarray*}
    R\Ho[\com_X]{G}{\omt X\ou[L]{\rho_!\D_X}\otimes\rho_! R\Gamma_{[X\setminus Z]}\mc M} & \simeq & \qquad\qquad\qquad\qquad \hspace{30mm}\phantom{a}
\end{eqnarray*}
\vspace{-10mm}
\nopagebreak
  \begin{eqnarray*} 
    \phantom{a}\qquad\qquad\qquad\qquad  &\simeq &    R\Ho[\com_X]{G}{\omt X\ou[L]{\rho_!\D_X}\otimes\rho_!(R\Gamma_{[X\setminus Z]}\mathcal{O}_X\ou[\mbb D]{}\otimes\mc M) }  \\\notag
    &  \simeq  &  R\Ho[\com_X]{G\otimes \com_{X\setminus Z}}{\omt X\ou[L]{\rho_!\D_X}\otimes\rho_!\mc M}  \\\notag
    &  \simeq  &  R\Ho[\com_X]{Rf_!\,f^{-1}(G\otimes\com_{X\setminus Z})}{\omt X\ou[L]{\rho_!\D_X}\otimes\rho_!\mc M} \\\notag
    &  \simeq  & Rf_* R\Ho[\com_{Y}]{f^{-1}(G\otimes\com_{X\setminus Z})}{f^!(\omt X\ou[L]{\rho_!\D_X}\otimes\rho_!\mc M)} \\\notag
    &  \simeq  & Rf_* R\Ho[\com_{Y}]{f^{-1}(G\otimes\com_{X\setminus Z})}{\omt Y\ou[L]{\rho_!\D_Y}\otimes\rho_!(\Dfinv f\mc M)} \ . \\\notag
  \end{eqnarray*}

  We have used Lemma \ref{lem:regular_twist} with $\mc
  R:=R\Gamma_{[X\setminus Z]}\mathcal{O}_X$ in the second
  isomorphism, the fact that $f|_{Y\setminus f^{-1}(Z)}$ is an isomorphism in the third
  isomorphism and Theorem \ref{thm:indsheaves}\emph{(ii)} in the last
  isomorphism.

\emph{(ii)}  
There exists a local system $L$ on
$U\setminus Z$ such that $Rf_!f^{-1}\com_{U\setminus Z}\simeq\com_{U\setminus
  Z}\oplus L$. We have

  \begin{eqnarray*} 
Rf_*R\Ho[\com_{Y}]{f^{-1}(\com_{U\setminus Z})}{\omt Y\ou[L]{\rho_!\D_Y}\otimes\rho_!(\Dfinv f\mc M)}    &\simeq &\qquad\qquad\qquad\qquad\qquad\qquad\phantom{a}
\end{eqnarray*}
\vspace{-10mm}
\begin{eqnarray*}
\phantom{a}\qquad\qquad\qquad\qquad\qquad  &  \simeq  &   
Rf_*R\Ho[\com_{Y}]{f^{-1}(\com_{U\setminus Z})}{f^!(\omt X\ou[L]{\rho_!\D_X}\otimes\rho_!\mc M)} 
\\
&  \simeq  &  
R\Ho[\com_X]{Rf_!\,f^{-1}(\com_{U\setminus Z})}{\omt X\ou[L]{\rho_!\D_X}\otimes\rho_!\mc M} 
\\
&  \simeq  &  
R\Ho[\com_X]{\com_{U\setminus Z}\oplus L}{\omt X\ou[L]{\rho_!\D_X}\otimes\rho_!\mc M}  
\\
&\simeq &   
R\Ho[\com_X]{\com_U\oplus L}{\omt X\ou[L]{\rho_!\D_X}\otimes\rho_!(R\Gamma_{[X\setminus Z]}\mathcal{O}_X \ou[\mbb D]{}\otimes \mc M )}  
\\
    &  \simeq  &  
R\Ho[\com_X]{\com_U\oplus L}{\omt X\ou[L]{\rho_!\D_X}\otimes\rho_! R\Gamma_{[X\setminus Z]}\mc M}   
\\
    &  \simeq  &  
R\Ho[\com_X]{\com_U\oplus L}{\omt X\ou[L]{\rho_!\D_X}\otimes\rho_! \mc M}  \ , 
  \end{eqnarray*}
  where we used Theorem \ref{thm:indsheaves} \emph{(ii)} in the first
  isomorphism, Lemma \ref{lem:regular_twist} with $\mc R:=\RGa{X\setminus
    Z}{\mathcal{O}_X}$ in the fourth isomorphism and Theorem
  \ref{thm:RGamma}\emph{(ii)} in the fifth isomorphism.

\end{proof}

In Subsection \ref{subsec:meromorphic_case} we will need also some
results on $\D_M$-modules for $M$ a real analytic manifold,
i.e. modules over the sheaf of rings of linear differential operators
of finite order with real analytic coefficients on $M$. We refer to
\cite[Section VII.1]{bjork}. Let us conclude with some notation about
$\D_M$-modules and their De Rham complexes. We denote with
$\Omega^\omega$ the sheaf on $M$ of differential forms of maximal
degree with real analytic coefficients. For $\mc M\in\Mod(\D_M)$, we
set
 $$ \drt[M]{\mc M}:=\rho_!\mc M\overset L{\underset{\rho_!\D_M}{\otimes}}(\Omega^\omega\underset{\mc C^\omega}\otimes\dbt_{M_sa})  \ . $$
 Even for such an object there are formulae similar to
 \eqref{eq:solutions_ot_thom}, \eqref{eq:inverse_image_modification},
 \eqref{eq:inverse_image_ramification} and those in the statement of
 Theorem \ref{thm:indsheaves} (see \cite{prelli_laplace,prelli_msmf}).

\subsection{Asymptotic decomposition of meromorphic
  connections}

In this subsection we are going to recall some fundamental results on
the asymptotic decompositions of flat meromorphic connections. The
first results on this subject were obtained by H. Majima
(\cite{majima_lnm}) and C. Sabbah (\cite{sabbah_aif}). In particular,
Sabbah proved that any \emph{good} formal decomposition of a flat
meromorphic connection on a complex surface admits an asymptotic lift
on small multisectors. Let us recall that Sabbah conjectured that any
flat meromorphic connection admits a \emph{good} formal decomposition
after a finite number of complex pointwise blow-up and ramification
maps (\cite{sabbah_ast}). Such a conjecture was proved in the
algebraic case by T. Mochizuki (\cite{mochizuki1,mochizuki2}) and in
the analytic one by K. Kedlaya (\cite{kedlaya1,kedlaya2}).  As in this
paper we are not concerned with the formal decomposition of flat
meromorphic connections and since the goodness property is not needed
within the scope of our results, we are not going to give details on
them. In this way, we will avoid to go into unessential technicalities
for the rest of the paper. For this subsection we refer to
\cite{sabbah_ast,mochizuki2,hien_periods_manifolds}.

Let $X$ be an analytic manifold of dimension $n$ and $Z$ a divisor of
$X$. Let us start by recalling some results about integrable
connections on $X$ with meromorphic poles on $Z$.  As in the rest of
the paper we will just need the case where $Z$ is a normal crossing
hypersurface, from now on, we will suppose such hypothesis.

Let $\Omega^j_X$ be the sheaf of $j$-forms on $X$. Let $\mc M$ be a
finitely generated $\mathcal{O}_X[*Z]$-module endowed with a $\com_X$-linear morphism
$\nabla:\mc M\to\Omega^1_X\otimes_{\mathcal{O}_X}\mc M$ satisfying the Leibniz
rule, that is to say, for any $h\in\mathcal{O}_X[*Z]$, $m\in\mc M$,
$\nabla(hm)=dh\otimes m+h\nabla m$. The morphism $\nabla$ induces
$\com_X$-linear morphisms $\nabla^{(j)}:\Omega^j_X\otimes_{\mathcal{O}_X}\mc
M\to \Omega^{j+1}_X\otimes_{\mathcal{O}_X}\mc M$.

\begin{df}\label{df:meromorphic_connection}
  A \emph{flat meromorphic connection on $X$ with poles along $Z$ } is
  a locally free $\mathcal{O}_X[*Z]$-module of finite type $\mc M$
  endowed with a $\com_X$-linear morphism $\nabla$ as above such that
  $\nabla^{(1)}\circ\nabla=0$.
\end{df}

In general, in the literature, a meromorphic connection is just a
coherent $\mc O[*Z]$-module. As stated in Proposition 1.2.1 of
\cite{sabbah_ast}, locally, the condition of being locally free is not
restrictive. Since in this paper we deal with the local study of
meromorphic connections and holonomic $\D$-modules, we adopted
Definition \ref{df:meromorphic_connection}. For sake of shortness, in
the rest of the paper, we will drop the adjective ``\emph{flat}''. If
there is no risk of confusion, given a meromorphic connection $(\mc
M,\nabla)$, we will simply denote it by $\mc M$.

Let $\mc M_1, \mc M_2$ be two locally free $\mathcal{O}_X[*Z]$-modules, a
morphism $\phi:\M_1\to\M_2$ induces a morphism
$\phi':\Omega^1_X\otimes_{\mathcal{O}_X}\mc
M_1\to\Omega^1_X\otimes_{\mathcal{O}_X}\mc M_2$. A \emph{morphism of
  meromorphic connections} $(\mc M_1,\nabla_1)\to(\mc M_2,\nabla_2)$
is given by a morphism $\phi:\M_1\to\M_2$ of locally free
$\mathcal{O}_X[*Z]$-modules such that
$\phi'\circ\nabla_1=\nabla_2\circ\phi$. We denote by $\mfr M(X,Z)$ the
\emph{category of meromorphic connections with poles along} $Z$.

Let us recall some facts on meromorphic connections that will be
useful in this paper, we refer to \cite{bjork} or Sections I.1.2 and
I.1.3 of \cite{sabbah_ast}.

It is well known that the category $\mfr M(X,Z)$ is equivalent to the
image in $\dmod[h]X$ of the functor $\cdot\overset{D}\otimes\mathcal{O}[*Z]$ on
the full subcategory of $\dmod[h]X$ consisting of objects with
singular locus contained in $Z$. In particular, if $\mc
M\in\dmod[h]{X}$ is such that $S(\mc M)$ is contained in $Z$ and $\mc
M\simeq R\Gamma_{[X\setminus Z]}\mc M\simeq\mc
M\Dotimes\mathcal{O}[*Z]$, then $\mc M$ is a locally free
$\mathcal{O}_X[*Z]$-module and the morphism $\nabla:\mc
M\lra\Omega^1_X\ou{\mathcal{O}_X}\otimes\mc M$, defined in a local coordinate
system $z:U\to \com^{n}$ by $\nabla m:=\ou[n]{j=1}\sum
dz^j\otimes\partial_{z^j}m$, gives rise to a meromorphic connection. A
meromorphic connection is said \emph{regular} if it is regular as a
$\D_X$-module. Furthermore the tensor product in $\mfr M(X,Z)$ is well
defined and it coincides with the tensor product of
$\D_X$-modules. With an abuse of language, given a holonomic
$\D_X$-module $\mc M$ with singular locus contained in $Z$, we will
call $\mc M$ a meromorphic connection if $\mc M\simeq
R\Gamma_{[X\setminus Z]}\mc M\simeq \mc M\Dotimes\mathcal{O}_X[*Z]$.

Let us recall that, if $(\mc M,\nabla)\in\mfr M(X,Z)$ and $\mc M$ is
an $\mathcal{O}_X[*Z]$-module of rank $r$, then, in a given basis of local
sections of $\mc M$, we can write $\nabla$ as $d-A$ where $A$ is a
$r\times r$ matrix with entries in
$\Omega^1_X\otimes_{\mathcal{O}_X}\mathcal{O}_X[*Z]$. Now, let $X'$ be a complex
manifold and $f:X'\to X$ a holomorphic map. Let us suppose that
$Z':=f^{-1}(Z)$ is a normal crossings hypersurface of $X'$ and that
$f$ is smooth on $X'\setminus Z'$. Considering $\mc M$ as a
$\D_X$-module, it satisfies $R\Gamma_{[X\setminus Z]}\mc M\simeq\mc
M$. Moreover, since, for $j\geq1$, $\mathrm{supp}\,H^j\Dfinv f\mc
M\subset Z'$, one checks that $\Dfinv f\mc M\simeq\dfinv f\mc M\simeq
R\Gamma_{[X'\setminus Z']}\dfinv f{\mc M}$. Hence $\dfinv f\mc M$ can
be considered as an object of $\mfr M(X',Z')$. As an $\mathcal
O_{X'}[*Z']$-module, it is isomorphic to $f^{-1}\mc M$ and the matrix
of the connection in a local base is $f^*A$.

Let us now introduce the elementary asymptotic decompositions.

Let us denote by $\widetilde X$ the real oriented blow-up of the
irreducible components of $Z$ and by $\pi:\widetilde X\to X$ the
composition of all these. Let us suppose that $Z$ is locally defined
by $x_1\cdot\ldots\cdot x_k=0$. Then, locally $\widetilde X\simeq
(S^1\times\rea_{\geq0})^k\times\com^{n-k}$. In what follows, $S^1$
will be identified with the unit circle in $\com$, so
$S^1=\{e^{i\theta}\in\com;\theta\in\rea\}$.

A \emph{multisector} is an open subset $S\subset\widetilde X$ such
that there exist $a_j,b_j\in\rea$, $a_j<b_j$, $r_j\in\reagz$
($j=1,\ldots,k$), $W$ an open neighborhood of $0\in\com^{n-k}$ such
that
\begin{equation}
  \label{eq_multisector}
  S=\Big(\ou[k]{j=1}\prod I_j\times[0,r_j[\Big)\times W\subset (S^1\times\rea_{\geq0})^k\times\com^{n-k}\ ,
\end{equation}
where $I_j=\{e^{i\theta}\in\com;\ \theta\in]a_j,b_j[\}\subset
S^1$. Given $(r,\tau)\in\reagz^2$, we say that a multisector $S$ is
\emph{small with respect to} $(r,\tau)$ if $S$ can be written as in
\eqref{eq_multisector} and, for any $j=1,\ldots,k$, $b_j-a_j<\tau$,
$r_j<r$ and for any $x\in W$, $|x|<r$.

Let $\overline x_1, \ldots, \overline x_n$ denote the antiholomorphic
coordinates on $X$. Then, for $j=k+1,\ldots,n$ (resp. $j=1,\ldots,k$)
$\partial_{\overline x_j}$ (resp. $\overline x_j\partial_{\overline
  x_j}$) acts on $\mc C^\infty_{\widetilde X}$ (see \cite{sabbah_aif}
2.12 or \cite{sabbah_ast} 1.1.4). The sheaf of algebras on $\widetilde
X$ of \emph{holomorphic functions with asymptotic development on $Z$},
denoted $\mc A_{\widetilde X}$, is defined as
$$ \mc A_{\widetilde X}:=\ou[k]{j=1}\bigcap\mrm{ker}\Big(\overline x_j\partial _{\overline x_j}:\mc C^\infty_{\widetilde
  X}\to\mc C^\infty_{\widetilde X}\Big)\cap\ou[n]{j=k+1}\bigcap
\mrm{ker}\Big(\partial _{\overline x_j}:\mc C^\infty_{\widetilde
  X}\to\mc C^\infty_{\widetilde X}\Big) \ .$$ The sections of $\mc
A_{\widetilde X}$ are holomorphic functions on $\widetilde X$ which
admit an asymptotic development as explained in Proposition B.2.1 of
\cite{sabbah_ast}. Moreover, $\mc A_{\widetilde X}$ is a
$\pi^{-1}\mathcal{O}_X$-module.

\begin{df}\label{df:elementary_model}
  \begin{enumerate}
  \item Let $\phi$ be a local section of
    $\mathcal{O}_X[*Z]/\mathcal{O}_X$, we denote by $\mc L^\phi$ the
    meromorphic connection of rank $1$ whose matrix in a basis is
    $d\phi$.
  \item An \emph{elementary local model} $\mc M$ is a meromorphic
    connection isomorphic to a direct sum
$$\ou{\alpha\in A}\oplus\mc L^{\phi_\alpha}\otimes\mc R_\alpha  \ , $$
where $A$ is a finite set, $(\phi_\alpha)_{\alpha\in A}$ is a family
of local sections of $\mathcal{O}_X[*Z]/\mathcal{O}_X$ and $(\mc
R_\alpha)_{\alpha\in A}$ is a family of regular meromorphic
connections.
\item We say that $(\mc M,\nabla)\in\mfr M(X,Z)$ admits an elementary
  $\mc A$-decomposition if there exist an elementary local model $(\mc
  M^{el},\nabla^{el})$ and $(r,\tau)\in\reagz^2$ such that for any
  multisector $S\subset\widetilde X$ small with respect to $(r,\tau)$,
  there exists $Y_S\in\mrm Gl(\mrm{rk}\,\mc M,\mc A_{\widetilde
    X}(S))$ such that the following diagram commutes
\begin{equation}
  \label{eq:asymptotic_connection}
\xymatrix{
      \pi^{-1}\mc M\ou{\pi^{-1}\mathcal{O}_X}\otimes\mc A_{\widetilde X}(S)
        \ar[r]^{\nabla\phantom{abcde}}\ar[d]_{\sim}^{Y_S\cdot}
     & \pi^{-1}(\mc M\ou{\mathcal{O}_X}{\otimes}\Omega^1_X)\ou{\pi^{-1}\mathcal{O}_X}{\otimes}\mc A_{\widetilde X}(S)
\ar[d]_{\sim}^{Y_S\cdot}
\\
 \pi^{-1}\mc M^{el}\ou{\pi^{-1}\mathcal{O}_X}\otimes\mc A_{\widetilde X}(S)
        \ar[r]^{\nabla^{el}\phantom{abcd}}
     & \pi^{-1}(\mc M^{el}\ou{\mathcal{O}_X}{\otimes}\Omega^1_X)\ou{\pi^{-1}\mathcal{O}_X}{\otimes}\mc A_{\widetilde X}(S)
\ .\\
}
\end{equation}
\end{enumerate}
\end{df}

Remark that the isomorphism in \eqref{eq:asymptotic_connection}
depends on $S$ and, in general, it does not give a global isomorphism.

The following Theorem is obtained by combining fundamental results of
C. Sabbah (\cite{sabbah_aif}), T. Mochizuki
(\cite{mochizuki1,mochizuki2}) and K. Kedlaya (\cite{kedlaya1,
  kedlaya2}).

\begin{thm}\label{thm:asymptotic_lift}
  Let $(\mc M, \nabla)\in\mfr M(X,Z)$. For any $x_0\in Z$ there exist
  a neighborhood $W$ of $x_0$ and a modification with respect to
  $W\cap Z$ above $x_0$, $\sigma:Y\to X$ such that $\sigma^{-1}(Z)$ is
  a normal crossing divisor and there exists a ramification map
  $\eta:X'\to Y$ fixing $\sigma^{-1}(Z)$ such that
  $\dfinv{(\eta\circ\sigma)}\mc M|_W$ admits an elementary $\mc
  A$-decomposition.
\end{thm}

\section{$\rea$-sa-constructibility of the tempered De Rham complex of
  holonomic $\D$-modules}
\sectionmark{Constructibility of tempered De Rham complex}

In this section we are going to prove the following

\begin{thm}\label{thm:main}
  Let $X$ be a complex analytic manifold, $\mc M\in\dbdx[h]$. Then
  $\drt [X]{\mc M}\in\bdc{\com_\xsa}$ is \reac.
\end{thm}

\begin{proof}

By using the distinguished triangle
$$R\Gamma_{[S(\mc M)]}\mc M\lra\mc M \lra R\Gamma_{[X\setminus S(\mc M)]}\mc M\overset{+1}\lra  \ ,$$

it turns out that it is sufficient to prove the
$\rea$-sa-constructibility of $\drt[X](R\Gamma_{[S(\mc M)]}\mc M)$ and
$\drt[X](R\Gamma_{[X\setminus S(\mc M)]}\mc M)$. We will deal these two
cases separately in the next subsections.


\end{proof}

\subsection{The case of $R\Gamma_{[X\setminus S(\mc M)]}\mc
  M$}\label{subsec:meromorphic_case}

In the present Subsection, first we prove the conjecture for
elementary models, then for connections admitting an elementary $\mc
A$-decomposition and in the end for meromorphic connections.

Recall that, for $Z$ a normal crossings hypersurface and
$\phi\in\mathcal{O}_X[*Z]/\mathcal{O}_X$, we denote by $\mc L^\phi$ the meromorphic
connection of rank $1$ whose matrix in a basis is $d\phi$. As a
$\D_X$-module, $\mc L^\phi$ is isomorphic to $\RGa{X\setminus
  Z}{(\D_X\exp(\phi))}$. Remark that, with a harmless abuse, we use
the same notation for the real analytic case.

\begin{lem}\label{lem:constructibility_elem_model}
  Consider $X:=\com^n$ with standard coordinates
  $(x_1,\ldots,x_n)$. Set $Z:=\{(x_1,\ldots,x_n)\in X;
  \ x_1\cdot\ldots\cdot x_n=0\}$. Given $ \phi\in\frac {\mathcal{O}(*Z)}{\mathcal{O}}$ and $\mc R\in\bdc[rh]{\D_X}$,
  we have that $\drt[X]{(\mc L^{\phi}\otimes \mc R)}$ is \reac.
\end{lem}

\begin{proof}
  First, let us remark that, by Lemma \ref{lem:regular_twist} it is
  sufficient to prove the statement with $\mc R\simeq\mathcal{O}_X$.  

  Given $U\in\Op^c(\com^n_{sa})$, by Theorem
  \ref{thm:hironaka_rectilinearization} we have that for any
  $x_0\in\rea^{2n}$ there exist real analytic map
  $g:\rea^{2n}\to\rea^{2n}$ and a compact set $K\subset\rea^{2n}$ such that
  \begin{enumerate}
  \item $g(K)$ is a neighborhood of $x_0$,
  \item $g^{-1}(U)$ is a union of quadrants,
  \item $g$ induces an open embedding of $g^{-1}(U)$ into $X$.
  \item $\psi(y):=\phi\circ g(y)=a\cdot y_1^{-k_1}\cdot\ldots\cdot
    y_{2n}^{-k_{2n}}$, $a\in\rea^\times$, $k_1,\ldots,k_{2n}\in\nat$.
  \end{enumerate}

Since $\overline U$ is compact, it is sufficient to prove that 
$$ R\Ho[\com_X]{\com_{U\cap g(K)}}{\drt[X]{\mc L^\phi}}\in\bdc[\rea-c]{\com_X}  \ .$$

Set $g_k:=g|_K$. We have the following sequence of isomorphisms
\begin{equation}
   \label{eq:functorial_rectilinearization}
  R\Ho[\com_X]{\com_{U\cap g(K)}}{\drt[X]{\mc L^\phi}}\simeq\qquad\qquad\qquad\qquad\qquad\qquad\qquad\qquad\phantom{a}
\end{equation} 
\vspace{-9mm}
\begin{eqnarray*}
\phantom{a}\qquad\qquad\qquad\qquad\qquad\qquad&\simeq &   R\Ho[\com_X]{Rg_{K!}g_K^{-1}\com_{U\cap g(K)}}{\drt[X]{\mc L^\phi}}\\\notag
&\simeq & Rg_{K*}  R\Ho[\com_{K}]{g_K^{-1}\com_{U\cap g(K)}}{g_K^!\drt[X]{\mc L^\phi}}\\\notag
&\simeq & Rg_{K*}  R\Ho[\com_{K}]{\com_{g^{-1}(U)\cap K}}{\drt[\rea^{2n}]{\Dfinv{g_K}\mc L^\phi}}\\\notag
&\simeq & Rg_{K*}  R\Ho[\com_{K}]{\com_{g^{-1}(U)\cap K}}{\drt[\rea^{2n}]{\mc L^\psi}} \ .
\end{eqnarray*}

Set $M:=\rea^{2n}$, $V:=g^{-1}(U)\cap K$ and
$Z':=\{(y_1,\ldots,y_{2n});\ y_1\cdot\ldots\cdot y_{2n}=0\}$. Clearly
one has
$$ R\Ho[\com_{M}]{\com_{V}}{\drt[M]{\mc L^\psi}}_{M\setminus Z'}\in\bdc[\rea-c]{\com_{M}}  \ .$$

Let us now consider, for $y_0\in Z'$
$$ R^j\Ho[\com_{M}]{\com_{V}}{\drt[M]{\mc L^\psi}}_{y_0} \ .$$
We are going to prove that it does not depend on $y_0$, this will
conclude the proof.

First, remark that 
\begin{eqnarray*}
  R\Ho[\com_{M}]{\com_{V}}{\drt[M]{\mc L^\psi}}&\simeq&R\Ho[\rho_!\D_M]{\com_{V}\otimes\com_{M\setminus Z'}\otimes\mbb D_M\mc L^\psi}{\dbt_M}\\
&\simeq&R\Ho[\rho_!\D_M]{\mc L^{-\psi}}{R\Ho{\com_{V}}{\dbt_M}}\\
&\simeq&R\Ho[\D_M]{\mc L^{-\psi}}{\THom{\com_{V}}{\Db_M}} \ .
\end{eqnarray*}

Now, 
\begin{eqnarray*}
R\Ho[\D_M]{\mc L^{-\psi}}{\THom(\com_{V},\Db_M)}_{y_0}& \simeq &\qquad\qquad\qquad\qquad \hspace{30mm}\phantom{a}
\end{eqnarray*}
\vspace{-10mm}
\nopagebreak
  \begin{eqnarray*} 
    \phantom{a}\qquad\qquad\qquad\qquad\qquad  &\simeq & 
\underset{y_0\in W}\varinjlim\ R\ho[\rho_!\D_W]{j_W^{-1}\mc L^{-\psi}}{j_W^{-1}\THom(\com_{V},\Db_M)}\\  
& \simeq &\underset{y_0\in W}\varinjlim\ R\ho[\rho_!\D_M]{\mc L^{-\psi}}{Rj_{W*}j_W^{-1}\THom(\com_{V},\Db_M)}\\  
& \simeq &\underset{y_0\in W}\varinjlim\ R\ho[\rho_!\D_M]{\mc L^{-\psi}}{\THom(\com_{V\cap W},\Db_M)}\\  
& \simeq &\underset{y_0\in W}\varinjlim\ R\ho[\com_M]{\com_{V\cap W}}{\drt[M]{\mc L^{\psi}}} \ , 
\end{eqnarray*}
where in the third isomorphism, we used the fact that, for
$W'\subset\subset W$, the restriction morphism $$ \THom(\com_{V\cap
  W},\Db_M)\lra\THom(\com_{V\cap W'},\Db_M) $$ factorizes through
$Rj_{W*}j_W^{-1}\THom(\com_{V},\Db_M)$, for $j_W:W\to M$ the
inclusion.

Consider the map $f:M\to\rea$,
$f(y_1,\ldots,y_{2n})=a^{-1}y_1^{k_1}\cdot\ldots\cdot y_{2n}^{k_{2n}}$ and set
$B:=B(y_0,\ep)$. We have the following series of isomorphisms

\begin{eqnarray*}
 R\ho[\com_M]{\com_{V\cap B}}{\drt[M]{\mc L^\psi}} & \simeq & R\ho[\com_M]{\com_{V\cap B}}{\drt[M]{\Dfinv f\mc L^{1/t}}}\\
& \simeq & R\ho[\com_M]{\com_{V\cap B}}{f^!\drt[\rea]{\mc L^{1/t}[1-2n]}}\\
& \simeq & R\ho[\com_\rea]{Rf_!\com_{V\cap B}[2n-1]}{\drt[\rea]{\mc L^{1/t}}}\\
& \simeq & R\ho[\com_\rea]{\com_{f(V\cap B)}}{\drt[\rea]{\mc L^{1/t}}} \\
& \simeq & R\ho[\com_\rea]{\mc L^{-1/t}}{\THom(\com_{f(V\cap B)},\Db_M)}  \ .
\end{eqnarray*}

Using the well known solvability of the homogeneous and non homogeneous
differential equations related to the operator $t^2\frac d{dt}-1$
acting on tempered distributions on the real line, we have that for
$\ep$ small enough the following isomorphisms hold
\begin{multline*}
R^j\ho[\D_{M}]{\mc L^{-\psi}}{\THom(\com_{V\cap B(y_0,\ep)},\Db_M)}\simeq\\
\simeq\begin{cases}
  0 & \text{ if } j>0\\
 0 &  \text{ if } j=0 \text{ and } \exp(-\psi)\notin \THom(\com_{V\cap B(y_0,\ep)},\Db_M)\\
 \com &  \text{ if } j=0 \text{ and } \exp(-\psi)\in \THom(\com_{V\cap B(y_0,\ep)},\Db_M) \ .
\end{cases}
\end{multline*}

\end{proof}

Let us now consider the case of meromorphic connections admitting an
elementary $\mc A$-decomposition.

\begin{lem}\label{lem:constructibility_A-decomp}
  Let $X$ be a complex analytic manifold and $\mc M\in\dmod[h]X$ be such that
  \begin{enumerate}
  \item $\SM$ is a normal crossing hypersurface,
  \item $\mc M\simeq R\Gamma_{[X\setminus \SM]}\mc M$,
  \item as a
    meromorphic connection, $\mc M$ admits an elementary $\mc A$-decomposition.  
  \end{enumerate} 
  Then $\drt[X]{\mc M}$ is \reac.
\end{lem}

\begin{proof}
  For sake of shortness, let us set $Z:=\SM$.

We want to prove that for any $G\in\bdc[\rea-c]{\com_X}$ 
\begin{equation}
  \label{eq:constructibility}
  R\Ho[\com_\xsa]G{\drt[X]{\mc M}}\in\bdc[\rea-c]{\com_X}\ .
\end{equation}

First, let us remark that it is sufficient to prove
\eqref{eq:constructibility} for $G=\com_U$ for $U\in\Op^c(\xsa)$.

We have the following sequence of isomorphisms

  \begin{eqnarray*}
    R\Ho[\com_X]{\com_U}{\omt X\ou[L]{\rho_!\D_X}\otimes\rho_! R\Gamma_{[X\setminus Z]}\mc M}  &\simeq &   R\Ho[\com_X]{\com_U}{\omt X\ou[L]{\rho_!\D_X}\otimes\rho_!( \mathcal{O}_X[*Z]\overset{\mbb D}\otimes\mc M)}  \\
    &  \simeq  &  R\Ho[\com_X]{\com_U\otimes \com_{X\setminus Z}}{\omt X\ou[L]{\rho_!\D_X}\otimes\rho_!\mc M}   \ .\\
  \end{eqnarray*}

  Where we used Lemma \ref{lem:regular_twist} with $\mc R:=\mathcal{O}_X[*Z]$
  in the second isomorphism.

  Hence, it is sufficient to prove the
  \eqref{eq:constructibility} for $G=\com_V$,
  $V\in\Op^c(\xsa)$, $V\subset X\setminus Z$.

  Let us briefly recall the meaning of the hypothesis \emph{(iii)} of
  the statement we are proving. The problem being local, we can
  suppose that $X\simeq \com^n$ and
  $Z\simeq\{(x_1,\ldots,x_n)\in\com^n;x_1\cdot\ldots\cdot
  x_k=0\}$. Furthermore, as $\mc M\simeq R\Gamma_{[X\setminus Z]}\mc
  M$ and as $V\subset X\setminus Z$ we can suppose that both $\mc M$
  and $\mbb D\mc M$ are meromorphic connections, in particular they
  are locally free $\mathcal{O}[*Z]$-modules of finite rank. 
  Let $\Omega^\bullet_X$ be the complex of differential forms on $X$.

  Recall that we defined $\widetilde X$ as the real oriented blow-up
  of the irreducible components of $Z$ and by $\pi:\widetilde X\to X$
  the composition of all these. Then, locally $\widetilde X\simeq
  (S^1\times\rea_{\geq0})^k\times\com^{n-k}$. Now, as $\mc M$ has an
  elementary $\mc A$-decomposition, there exist an elementary model
  $(\mc M^{el},\nabla^{el})$ and $(r,\tau)\in\reagz^2$ such that for
  any multisector $S\subset\widetilde X$ small with respect to
  $(r,\tau)$ there exists $Y_S\in\mathrm Gl(\mathrm{rk}\,\mc M,\mc
  A_{\widetilde X}(S))$ giving an isomorphism of complexes
$$
\xymatrix{
      0\ar[r] 
      & \pi^{-1}\mc M\ou{\pi^{-1}\mathcal{O}_X}\otimes\mc A_{\widetilde X}(S)
        \ar[r]^{\nabla\phantom{abcde}}\ar[d]^{Y_S\cdot}_{\sim}
     & \pi^{-1}(\mc M\ou{\mathcal{O}_X}{\otimes}\Omega^1_X)\ou{\pi^{-1}\mathcal{O}_X}{\otimes}\mc A_{\widetilde X}(S)
        \ar[r]^{\phantom{abcdefgasdfas}\nabla}\ar[d]^{Y_S\cdot}_{\sim}
     &     \ldots\\
        0\ar[r] 
     & \pi^{-1}\mc M^{el}\ou{\pi^{-1}\mathcal{O}_X}\otimes\mc A_{\widetilde X}(S)
        \ar[r]^{\nabla^{el}\phantom{abcd}}
     & \pi^{-1}(\mc M^{el}\ou{\mathcal{O}_X}{\otimes}\Omega^1_X)\ou{\pi^{-1}\mathcal{O}_X}{\otimes}\mc A_{\widetilde X}(S)
        \ar[r]^{\phantom{abcdefgasdfasfe}\nabla^{el}}
&        \ldots \ .\\
}
$$

Clearly, for any $(r,\tau)\in\reagz^2$, there exists a finite family
$\{S_k\}_{k\in K}$ of multisectors small with respect to $(r,\tau)$
such that, for any $L\subset K$, $\cap_{l\in L}S_l$ is a multisector
small with respect to $(r,\tau)$ and $\cup_{k\in K} S_k$ is an open
neighborhood of $\pi^{-1}(Z)\subset\widetilde X$.

Coming back to the proof of \eqref{eq:constructibility}, as
said above it is sufficient to prove it for $G=\com_V$,
$V\in\Op^c(\xsa)$, $V\subset X\setminus Z$. Given such a $V$, consider
$V_k:=V\cap\pi(S_k)$ for $\{S_k\}_{k\in K}$ a family of multisectors
as above, then there exists $W\in\Op^c(\xsa)$ such that $\overline
W\cap Z=\varnothing$ and $V=\cup_{k\in K}V_k\cup W$. Now, since for
any $L\subset K$, $\cap_{l\in L}V_l$ is contained in a multisector
small with respect to $(r,\tau)$, it follows that it is sufficient to prove \eqref{eq:constructibility} for
$G=\com_V$, $V\in\Op^c(\xsa)$, $V\subset X\setminus Z$ and
$\pi^{-1}(V)$ contained in a multisector $S$ small with respect to
$(r,\tau)$.

Recall that we denote by $X_\rea$ the real analytic manifold
underlying $X$, by $\mc C^\omega_{X_\rea}$ the sheaf of real analytic
functions on $X_\rea$. Furthermore, we denote by $\Omega^{\bullet,\bullet}_{X_\rea}$ the
double complex of sheaves of differential forms on $X_\rea$ with
coefficients in $\mc C^\omega_{X_\rea}$.

In Subsection \ref{subsection:tempered_de_rham}, in
\eqref{eq:de_rham_explicit}, we proved that
$R\Ho[\com_X]{\com_V}{\drt[X]{\mc M}}$ is isomorphic to the total
complex relative to \eqref{eq:de_rham_explicit}.

Now, as $\THom(\com_{\pi(S)\setminus Z},\Db)$ is a $\pi_*\mc A_{\widetilde
  X}|_S$-module, it makes sens to consider $Y_S$ as an isomorphism
$$ \THom(\com_{V},\Db)\ou{\mc C^\omega_{X_\rea}}{\otimes} \Omega_{X_\rea}^{j,k}\ou{\mathcal{O}_X}{\otimes}\mbb D\mc M\ou[{ Y_S\cdot}]{\sim}\lra 
\THom(\com_{V},\Db)\ou{\mc C^\omega_{X_\rea}}{\otimes}
\Omega_{X_\rea}^{j,k}\ou{\mathcal{O}_X}{\otimes}\mbb D\mc M^{el} \ .
$$
which satisfies the natural commuting conditions with respect to
$\nabla$, $\nabla^{el}$ and $\overline\partial$.

It follows that the total complex of \eqref{eq:de_rham_explicit} is
isomorphic to the total complex relative to the double complex
obtained from \eqref{eq:de_rham_explicit} replacing $\mc M$ with $\mc
M^{el}$ and $\nabla$ with $\nabla^{el}$.

The total complex of this last double complex is isomorphic to
$R\Ho[\com_X]{\com_V}{\drt[X]{(\mc M^{el})}}$.
 
By Lemma \ref{lem:constructibility_elem_model}, $\drt[X]{(\mc M^{el})}$ is
\reac. Hence we have that $\drt[X]{\mc M}$ is \reac\ too.

\end{proof}

\begin{prop}\label{prop:constructibility_connections}
  Let $X$ be a complex analytic manifold, $\mc M\in\dbdx[h]$. Then
  $\drt[X]{(\RGa{X\setminus S(\mc M)}{\mc M})}$ is \reac.
\end{prop}

\begin{proof}
First, let us suppose that $S(\mc M)$ is a hypersurface.

Locally on $X$, there exists a finite sequence of complex blow-up maps
$\pi:X'\to X$ such that, denoting $Z':=\pi^{-1}(S(\mc M))$,
$\pi|_{X'\setminus Z'}$ is a biholomorphism and $Z'$ is a normal
crossings hypersurface. By Proposition \ref{prop:constructibility_inv_im}
\emph{(i)} and Proposition \ref{prop:temp_supp_inv_image}, we have
that it is sufficient to prove the statement for $\mc M\in\dbd[h]X$ such that $S(\mc M)$ is a normal
crossings hypersurface and $R\Gamma_{[X\setminus\SM]}\mc M\simeq \mc
M$.

Now, as $\mc M$ is a bounded complex, by using inductively the
distinguished triangles $(1.7.3)$ or $(1.7.4)$ of \cite{ks_som}, one
checks easily that it is sufficient to prove the statement for $\mc M\in\dmod[h]X$ such that $S(\mc M)$ is a normal
crossings hypersurface and $R\Gamma_{[X\setminus \SM]}\mc M\simeq \mc
M$.

Now, $\mc M$ can be considered as a meromorphic connection. By Theorem
\ref{thm:asymptotic_lift}, locally on $X$, there exists a composition
of modifications and ramification maps $\pi:X'\to X$ such that $\dfinv
\pi\mc M$ admits an elementary $\mc A$-decomposition. Hence, by
Proposition \ref{prop:constructibility_inv_im} \emph{(i)} and \emph{(ii)}
it is sufficient to prove the statement for $\mc M\in\dmod[h]X$ such that $S(\mc M)$ is a normal crossings
hypersurface, $R\Gamma_{[X\setminus S(\mc M)]}\mc M\simeq \mc M$ and
it admits an elementary $\mc A$-decomposition.

This last case follows by Lemma \ref{lem:constructibility_A-decomp}.

To conclude the proof of Proposition
\ref{prop:constructibility_connections}, let us consider the case of $S(\mc
M)$ an analytic set. 

Locally, there exist hypersurfaces $H_1,\ldots,H_r$ such that $S(\mc
M)\simeq H_1\cap\ldots\cap H_r$. For sake of shortness, set
$V_j:=X\setminus H_j$, hence $X\setminus S(\mc M)\simeq
V_1\cup\ldots\cup V_r$. Let us prove the \reacy\ of
$\drt[X]{(\RGa{V_1\cup\ldots\cup V_r}{\mc M})}$ by induction on
$r$. The case $r=1$ has been treated above.

For $r>1$, we have the following distinguished triangle,

\begin{equation}
  \label{eq:dt_V_r}  
  \RGa{V_1\cup\ldots\cup V_r}{\mc M}\to\RGa{V_1\cup\ldots\cup V_{r-1}}{\mc M}\oplus\RGa{V_r}{\mc M}\to\RGa{(V_1\cap V_r)\cup\ldots\cup (V_{r-1}\cap V_r)}{\mc M}\overset{+1}\to \ . 
\end{equation}

By the inductive hypothesis, the second and the third terms of
\eqref{eq:dt_V_r} have \reac\ tempered De Rham complexes. It follows
that $\drt[X]{(\RGa{X\setminus S(\mc M)}{\mc M})}$ is \reac.

\end{proof}


\subsection{The case of $R\Gamma_{[\SM]}\mc M$}\label{subsec:finiteness_support}

In this subsection we are going to prove the following

\begin{prop}\label{prop:finiteness_support}
  Let $X$ be an analytic manifold, $\mc M\in\dbdx[h]$. Then
  $\drt[X]{(\RGa{\SM}{\mc M})}$ is \reac.
\end{prop}

\begin{proof}
  Since the problem is local, we can add (or neglect when necessary)
  the additional hypothesis that the singular locus of $\mc M$ is
  connected.

  For sake of shortness, set $Z:=\SM$. 

Let us prove our result with an induction on $d:=\dim\,Z$.

Clearly, for any complex manifold $Y$ and any $\mc N\in\bdc[h]{\D_Y}$
such that $\dim S(\mc N)=0$, $\drt[Y]{\RGa{S(\mc N)}{\mc N}}$ is
\reac.

Now, suppose $d:=\dim\,Z>0$ and that for any complex manifold $Y$ and
any $\mc N\in\bdc[h]{\D_Y}$ such that $\dim S(\mc N)<d$,
$\drt[Y]{\RGa{S(\mc N)}{\mc N}}$ is \reac. By Proposition
\ref{prop:constructibility_connections}, it follows that $\drt[Y]{(\mc N)}$
is \reac.

  We want to prove that, for any $G\in\bdc[\rea-c]{\com_X}$ 
  \begin{equation}
    \label{eq:constructibility_support}
    R\Ho[\com_X]{G}{\drt[X]{(\RGa{Z}{\mc M})}}\in\bdc[\rea-c]{\com_X} . 
  \end{equation}

  As the condition is local on $X$ and as $\bdc[\rea-c]{\com_X}$ is
  generated by objects of the form $\com_U$, for $U\in\Op(\xsa)$, it
  is sufficient to prove \eqref{eq:constructibility_support} with
  $G\simeq \com_U$, $U\in\Op^c(\xsa)$.

First, we have the following sequence of isomorphisms in
$\bdc{\com_X}$

\begin{eqnarray*}
  R\Ho[\com_X]{\com_U}{\omt X\ou[L]{\rho_!\D_X}\otimes\rho_!\RGa{Z}{\mc M}}  & \simeq & \qquad\qquad\qquad\qquad \hspace{30mm}\phantom{a}
\end{eqnarray*}
\vspace{-12mm}
\nopagebreak
  \begin{eqnarray*} 
    \phantom{a}\qquad\qquad\qquad\qquad\qquad &  \simeq  &  R\Ho[\com_X]{\com_U}{\omt X\ou[L]{\rho_!\D_X}\otimes\rho_!(\RGa{Z}{\mathcal{O}_X}\overset{\mbb D}\otimes\RGa{Z}{ \mc M})} \\
&  \simeq  &  R\Ho[\com_X]{\com_U\otimes \com_Z}{\omt X\ou[L]{\rho_!\D_X}\otimes\rho_!\RGa{Z}{\mc M}}  \\
&  \simeq  &  R\Ho[\com_X]{\com_{U\cap Z}}{\omt X\ou[L]{\rho_!\D_X}\otimes\rho_!\RGa{Z}{\mc M}} \ . \\
\end{eqnarray*}

We used Theorem \ref{thm:RGamma} \emph{(ii)} and the fact that $\RGa
Z{\RGa Z{\mc M}}\simeq \RGa Z{\mc M}$ in the first isomorphism and
Lemma \ref{lem:regular_twist} with $\mc R:=R\Gamma_{[Z]}\mathcal{O}_X$
in the second isomorphism.

It follows that it is sufficient to prove
\eqref{eq:constructibility_support} with $G\simeq\com_{U\cap Z}$ for all
$U\in\Op^c(\xsa)$.

Now, we are going to use some results of analytic geometry concerning
analytic sets, their regular and singular parts and their strict
transform (see \cite{hironaka_resolution_singularities} or
\cite{bm_local_resolutions_singularities}); for the convenience of the
reader we will briefly recall them here. Given an analytic set
$Z\subset X$, a point $z\in Z$ is said \emph{regular} if there exists
a neighborhood $W$ of $z$ such that $Z\cap W$ is a closed submanifold
of $W$. It is well known that the set of regular points of an analytic
set $Z\subset X$, denoted $Z_{r}$, is a dense open subset of
$Z$. Furthermore $Z_{s}:=Z\setminus Z_{r}$ is a closed analytic set
called the \emph{singular part of $Z$} whose dimension is smaller that
the dimension of $Z$. Moreover, there exists a proper morphism of
analytic manifolds $\pi:X'\to X$ and a closed submanifold $Z'$ of
$X'$, called the \emph{strict transform of} $Z$, such that
\begin{equation}
   \label{eq:modification_reg}
 \pi|_{Z'\setminus\pi^{-1}(Z_s)}:Z'\setminus\pi^{-1}(Z_s)\lra Z\setminus Z_s  
\end{equation}
is an isomorphism.

Now, it is sufficient to prove \eqref{eq:constructibility_support} with
$G\simeq\com_{U\cap Z_{r}}$ and $G\simeq\com_{U\cap Z_{s}}$ for
$U\in\Op^c(\xsa)$. Let us treat these two cases separately.

\emph{The case of $\com_{U\cap Z_{r}}$.}

Let $\pi:X'\to X$ be the map described above in
\eqref{eq:modification_reg} and let $Z'$ be the strict transform of
$Z$, set $\pi_{Z'}:=\pi|_{Z'}$. Clearly, $V:=\pi^{-1}(U\cap Z_{r})$ is
a relatively compact subanalytic subset of $Z'$ such that $\com_{U\cap
  Z_{r}}\simeq R\pi_{Z'\,!}\com_V$.

We have the following sequence of
isomorphisms
\begin{eqnarray}\label{eq:Z_r}
    R\Ho[\com_X]{\com_{U\cap Z_{r}}}{\omt X\ou[L]{\rho_!\D_X}\otimes\rho_! \RGa{Z}{\mc M}} & \simeq & \qquad \hspace{30mm}\phantom{a}
\end{eqnarray}
\vspace{-5mm}
\nopagebreak
  \begin{eqnarray*} 
    \phantom{a}\qquad\qquad\qquad\qquad\qquad  &\simeq &   R\Ho[\com_X]{R\pi_{Z'\,!}\,\com_V}{\omt X\ou[L]{\rho_!\D_X}\otimes\rho_!\RGa{Z}{\mc M}} \\\notag
    &  \simeq  &  R\pi_*R\Ho[\com_{Z'}]{\com_V}{{\pi_{Z'}}^!(\omt X\ou[L]{\rho_!\D_X}\otimes\rho_!\RGa{Z}{\mc M})} \\\notag
    &  \simeq  &  R\pi_*R\Ho[\com_{Z'}]{\com_{V}}{\omt {Z'}\ou[L]{\rho_!\D_{Z'}}\otimes\rho_!(\Dfinv {\pi_{Z'}}\RGa{Z}{\mc M})} \ . \\\notag
  \end{eqnarray*}
  We used Theorem \ref{thm:indsheaves}(ii) in the third isomorphism.

  Since $\dim S(\Dfinv {\pi_{Z'}}\RGa{Z}{\mc M})<\dim\,
  Z'=d$, by the inductive hypothesis we have that $\drt[Z']{(\Dfinv
    {\pi_{Z'}}\RGa{Z}{\mc M})}$ is \reac.

\emph{The case of $\com_{U\cap Z_{s}}$.}

We have the following sequence of isomorphisms

\begin{eqnarray*}
  R\Ho[\com_X]{\com_{U\cap Z_s}}{\omt X\ou[L]{\rho_!\D_X}\otimes\rho_!\RGa{Z}{\mc M}} 
 & \simeq & \qquad\qquad\qquad\qquad \hspace{30mm}\phantom{a}
\end{eqnarray*}
\vspace{-8mm}
\nopagebreak
  \begin{eqnarray*} 
    \phantom{a}\qquad\qquad\qquad\qquad\qquad  &\simeq&
  R\Ho[\com_X]{\com_U\otimes \com_{Z_s}}{\omt X\ou[L]{\rho_!\D_X}\otimes\rho_!\RGa{Z}{\mc M}}\\
  &\simeq&
  R\Ho[\com_X]{\com_U}{\omt X\ou[L]{\rho_!\D_X}\otimes\rho_!(\RGa{Z_s}{\mathcal{O}_X}\overset{\mbb D}\otimes\RGa{Z}{ \mc M})} \\
  &\simeq&
  R\Ho[\com_X]{\com_U}{\omt X\ou[L]{\rho_!\D_X}\otimes\rho_!\RGa{Z_s}{\mc M}} \ ,
\end{eqnarray*}
where we used Lemma \ref{lem:regular_twist} with $\mc
R:=\RGa{Z_s}{\mathcal{O}_X}$ in the second
isomorphism and Theorem \ref{thm:RGamma}\emph{(iii)} in the third isomorphism.

Since $S(\RGa{Z_s}{\mc M})\subseteq Z_s$ and its dimension is strictly
smaller than $d$, the conclusion follows by the inductive hypothesis.

\end{proof}


\addcontentsline{toc}{section}{


\textbf{References}}
\begin{thebibliography}{10}

\bibitem{bm_ihes}
E.~Bierstone and P.~D. Milman.
\newblock Semianalytic and subanalytic sets.
\newblock {\em Inst. Hautes \'Etudes Sci. Publ. Math.}, (67):5--42, 1988.

\bibitem{bm_local_resolutions_singularities}
E.~Bierstone and P.~D. Milman.
\newblock Local resolution of singularities.
\newblock In {\em Real analytic and algebraic geometry ({T}rento, 1988)},
  volume 1420 of {\em Lecture Notes in Math.}, pages 42--64. Springer, Berlin,
  1990.

\bibitem{bjork}
J.-E. Bj{\"o}rk.
\newblock {\em Analytic {${\mathcal D}$}-modules and applications}, volume 247
  of {\em Mathematics and its Applications}.
\newblock Kluwer Academic Publishers Group, Dordrecht, 1993.

\bibitem{bcr}
J.~Bochnak, M.~Coste, and M.-F. Roy.
\newblock {\em Real algebraic geometry}, volume~36 of {\em Ergebnisse der
  Mathematik und ihrer Grenzgebiete (3) [Results in Mathematics and Related
  Areas (3)]}.
\newblock Springer-Verlag, Berlin, 1998.
\newblock Translated from the 1987 French original, Revised by the authors.

\bibitem{dagnolo_kashiwara_rh}
A.~D'Agnolo and M.~Kashiwara.
\newblock Riemann--{H}ilbert correspondence for holonomic {${\mathcal
  D}$}-modules.
\newblock arXiv:1311.2374.

\bibitem{dagnolo_kashiwara_dim1}
A.~D'Agnolo and M.~Kashiwara.
\newblock On a reconstruction theorem for holonomic systems.
\newblock {\em Proc. Japan Acad. Ser. A Math. Sci.}, 88(10):178--183, 2012.

\bibitem{dmr}
P.~Deligne, B.~Malgrange, and J.-P. Ramis.
\newblock {\em Singularit\'es irr\'eguli\`eres}.
\newblock Documents Math\'ematiques (Paris) [Mathematical Documents (Paris)],
  5. Soci\'et\'e Math\'ematique de France, Paris, 2007.
\newblock Correspondance et documents. [Correspondence and documents].

\bibitem{hien_periods_manifolds}
M.~Hien.
\newblock Periods for flat algebraic connections.
\newblock {\em Invent. Math.}, 178(1):1--22, 2009.

\bibitem{hironaka_resolution_singularities}
H.~Hironaka.
\newblock Resolution of singularities of an algebraic variety over a field of
  characteristic zero. {I}, {II}.
\newblock {\em Ann. of Math. (2) 79 (1964), 109--203; ibid. (2)}, 79:205--326,
  1964.

\bibitem{hironaka_pisa}
H.~Hironaka.
\newblock {\em Introduction to real-analytic sets and real-analytic maps}.
\newblock Istituto Matematico ``L. Tonelli'' dell'Universit\`a di Pisa, Pisa,
  1973.
\newblock Quaderni dei Gruppi di Ricerca Matematica del Consiglio Nazionale
  delle Ricerche.

\bibitem{hironaka_subanalytic_sets}
H.~Hironaka.
\newblock Subanalytic sets.
\newblock In {\em Number theory, algebraic geometry and commutative algebra, in
  honor of {Y}asuo {A}kizuki}, pages 453--493. Kinokuniya, Tokyo, 1973.

\bibitem{kashiwara_regular1}
M.~Kashiwara.
\newblock On the maximally overdetermined system of linear differential
  equations. {I}.
\newblock {\em Publ. Res. Inst. Math. Sci.}, 10:563--579, 1974/75.

\bibitem{kashiwara_79}
M.~Kashiwara.
\newblock Faisceaux constructibles et syst\`emes holon\^omes d'\'equations aux
  d\'eriv\'ees partielles lin\'eaires \`a points singuliers r\'eguliers.
\newblock In {\em S\'eminaire Goulaouic-Schwartz, 1979--1980 (French)}, pages
  Exp. No. 19, 7. \'Ecole Polytech., Palaiseau, 1980.

\bibitem{kashiwara_riemann-hilbert}
M.~Kashiwara.
\newblock The {R}iemann-{H}ilbert problem for holonomic systems.
\newblock {\em Publ. Res. Inst. Math. Sci.}, 20(2):319--365, 1984.

\bibitem{kashiwara_dmod}
M.~Kashiwara.
\newblock {\em {$\mathcal D$}-modules and microlocal calculus}, volume 217 of
  {\em Translations of Mathematical Monographs}.
\newblock American Mathematical Society, Providence, RI, 2003.
\newblock \phantom{Translated from the 2000 Japanese original by Mutsumi Saito,
  Iwanami Series in Modern Mathematics}.

\bibitem{ks_som}
M.~Kashiwara and P.~Schapira.
\newblock {\em Sheaves on manifolds}, volume 292 of {\em Grundlehren der
  Mathematischen Wissenschaften}.
\newblock Springer-Verlag, Berlin, 1990.

\bibitem{ks_moderate_formal_cohomology}
M.~Kashiwara and P.~Schapira.
\newblock Moderate and formal cohomology associated with constructible sheaves.
\newblock {\em M\'em. Soc. Math. France (N.S.)}, (64):iv+76, 1996.

\bibitem{ks_indsheaves}
M.~Kashiwara and P.~Schapira.
\newblock Ind-sheaves.
\newblock {\em Ast\'erisque}, (271):136, 2001.

\bibitem{ks_microlocal_indsheaves}
M.~Kashiwara and P.~Schapira.
\newblock Microlocal study of ind-sheaves. {I}. {M}icro-support and regularity.
\newblock {\em Ast\'erisque}, (284):143--164, 2003.

\bibitem{kedlaya1}
K.~Kedlaya.
\newblock Good formal structures on flat meromorphic connections, {I}:
  Surfaces.
\newblock {\em Duke Math. J.}, 154(2):343--418, 2010.

\bibitem{kedlaya2}
K.~Kedlaya.
\newblock Good formal structures for flat meromorphic connections, {II}:
  Excellent schemes. 
\newblock {\em J. Amer. Math. Soc.}, 24(1):183--229, 2011.

\bibitem{majima_lnm}
H.~Majima.
\newblock {\em Asymptotic analysis for integrable connections with irregular
  singular points}, volume 1075 of {\em Lecture Notes in Mathematics}.
\newblock Springer-Verlag, Berlin, 1984.

\bibitem{malgrange_birkhauser}
B.~Malgrange.
\newblock {\em \'{E}quations diff\'erentielles \`a coefficients polynomiaux},
  volume~96 of {\em Progress in Mathematics}.
\newblock Birkh\"auser Boston Inc., Boston, MA, 1991.

\bibitem{mochizuki1} 
T.~Mochizuki.  
\newblock Good formal structure for meromorphic f\mbox{}lat
connections on smooth projective surfaces.
\newblock In {\em Algebraic analysis and around}, Adv. Stud. Pure
Math., 54, pages 223--253, Math. Soc. Japan, Tokyo, 2009.

\bibitem{mochizuki2}
T.~Mochizuki.
\newblock Wild harmonic bundles and wild pure twistor {$\mathcal D$}-modules.
\newblock {\em Ast\'erisque}, (340):x+607, 2011.

\bibitem{morando_existence_theorem}
G.~Morando.
\newblock An existence theorem for tempered solutions of {$\mathcal D$}-modules
  on complex curves.
\newblock {\em Publ. Res. Inst. Math. Sci.}, 43, 2007.

\bibitem{morando_preconstructibility}
G.~Morando.
\newblock Preconstructibility of tempered solutions of holonomic {$\mathcal
  D$}-modules.
\newblock {\em International Mathematics Research Notices},
  doi:10.1093/imrn/rns247, 2012.

\bibitem{prelli_msmf}
L.~Prelli.
\newblock Microlocalization of subanalytic sheaves.
\newblock {\em To appear as M\'emoires de la SMF 135}, arXiv:math/0702459.

\bibitem{prelli_subanalytic_sheaves}
L.~Prelli.
\newblock Sheaves on subanalytic sites.
\newblock {\em Rend. Semin. Mat. Univ. Padova}, 120:167--216, 2008.

\bibitem{prelli_laplace}
L.~Prelli.
\newblock Conic sheaves on subanalytic sites and {L}aplace transform.
\newblock {\em Rend. Semin. Mat. Univ. Padova}, 125:173--206, 2011.

\bibitem{sabbah_aif}
C.~Sabbah.
\newblock \'{E}quations diff\'erentielles \`a points singuliers irr\'eguliers
  en dimension {$2$}.
\newblock {\em Ann. Inst. Fourier (Grenoble)}, 43(5):1619--1688, 1993.

\bibitem{sabbah_ast}
C.~Sabbah.
\newblock \'{E}quations diff\'erentielles \`a points singuliers irr\'eguliers
  et ph\'enom\`ene de {S}tokes en dimension 2.
\newblock {\em Ast\'erisque}, (263):viii+190, 2000.

\bibitem{sabbah_lnm}
C.~Sabbah.
\newblock {\em Introduction to {S}tokes structures}, volume 2060 of {\em
  Lecture Notes in Mathematics}.
\newblock Springer, Heidelberg, 2013.

\end{thebibliography}

\vspace{10mm}

{\small{\sc Giovanni Morando
\medskip

Dipartimento di Matematica Pura ed Applicata,

Universit{\`a} degli Studi di Padova,

Via Trieste 63, 35121 Padova, Italy.}

E-mail address: $\texttt{gmorando@math.unipd.it}$
\medskip

{\sc Lehrstuhl f\"ur Algebra und Zahlentheorie

Universit\"atsstra\ss e 14

86159 Augsburg, Germany}

E-mail address: $\texttt{giovanni.morando@math.uni-augsburg.de}$}

\end{document}